\documentclass[12pt]{amsart}

\usepackage{amssymb}

\usepackage{enumerate}

\usepackage{graphicx}

\makeatletter
\@namedef{subjclassname@2010}{%
  \textup{2010} Mathematics Subject Classification}
\makeatother

\newtheorem{thm}{Theorem}[section]
\newtheorem{thma}{Theorem}

\newtheorem{lem}[thm]{Lemma}
\newtheorem{prop}[thm]{Proposition}

\newtheorem{lema}{Lemma}

\theoremstyle{definition}
\newtheorem{defin}[thm]{Definition}
\newtheorem{rem}[thm]{Remark}
\newtheorem{exa}[thm]{Example}

\numberwithin{equation}{section}

\DeclareMathOperator*{\esssup}{ess\,sup}

\begin{document}


\baselineskip=17pt



\title[Sharp Weighted Bounds]{Sharp Weighted Bounds for One--sided and Multiple Integral Operators$^*$}

\author[V. Kokilashvili]{Vakhtang Kokilashvili}
\address{Department of Mathematical Analysis, A. Razmadze Mathematical Institute, I. Javakhishvili Tbilisi State University, 2. University Str., 0186 Tbilisi, Georgia.}
\address{International Black Sea University, 3 Agmashenebeli Ave., Tbilisi 0131, Georgia.}
\email{kokil@rmi.ge}

\author[A. Meskhi]{Alexander Meskhi$^*$}
\address{Department of Mathematical Analysis, A. Razmadze Mathematical Institute, I. Javakhishvili Tbilisi State University, 2. University Str., 0186 Tbilisi, Georgia.}
\address{Department of Mathematics,  Faculty of Informatics and Control Systems, Georgian Technical University, 77, Kostava St., Tbilisi, Georgia.}
\email{meskhi@rmi.ge}

\author[M. A. Zaighum]{Muhammad Asad Zaighum}
\address{Abdus Salam School of Mathematical Sciences, GC University, 68-B New Muslim Town, Lahore, Pakistan.}
\email{asadzaighum@gmail.com}

\thanks{These results were presented in {\em Preprints of Abdus Salam School of Mathematical Sciences, Lahore, 22 January, 2013,} Preprint No. 498, see www.sms.edu.pk/journals/preprint/pre\_498.pdf}
\date{}

\begin{abstract}
In this paper we establish sharp weighted bounds (Buckley type theorems) for one--sided maximal and fractional integral operators in terms of one--sided $A_{p}$ characteristics. Appropriate sharp bounds for strong maximal functions, multiple potentials and singular integrals are derived. 
\end{abstract}



\maketitle
\section{Introduction}

One of the main problem in Harmonic analysis is to characterize a weight $w$ for which a given integral operator is bounded in $L^p_{w}$. An important class of such weights is the well-known $A_{p}$ class. It is known that $A_p$ condition is necessary and sufficient for the boundedness of Hardy--Littlewood and singular integral operators see, e.g., (\cite{COF-RFEFF}, \cite{HUNT-MUCK}, \cite{MUCK1}). However, the sharp dependence of the corresponding $L^p_{w}$ norms in terms of $A_{p}$ characteristic of $w$ is known only for some operators. The interest in the sharp weighted norm for singular integral operators is motivated by applications in partial differential equations (see e.g \cite{AST-IWA-SAK}, \cite{PET-VOL}, \cite{PET1}, \cite{PET2}). In this paper,  new sharp weighted estimates for one--sided maximal functions, one--sided fractional integrals, strong maximal functions, singular and fractional integrals with product kernels are derived.

Let $w$ be an almost everywhere positive locally integrable function on a subset $\Omega$ of $\mathbb{R}^n$.

 We denote by $L_{w}^{p}(\Omega)$, $1<p<\infty$, the set of all measurable functions $f:\Omega \rightarrow\mathbb{R}$ for which the norm
\begin{equation*}
\|f\|_{L_{w}^{p}(\Omega)}=\bigg(\int\limits_{\Omega}|f(x)|^{p}w(x)dx\bigg)^{\frac{1}{p}}
\end{equation*}
is finite. If $w\equiv$ const, then we denote $L_{w}^{p}(\Omega)=L^{p}(\Omega)$. Suppose that $L^{p,\infty}_{w}(\Omega)$ be the weighted weak Lebesgue space defined by the quasi norm:
 $$\|f\|_{L^{p,\infty}_{w}(\Omega)}=\sup\limits_{\lambda>0}\lambda[w(\{x\in\Omega: |f(x)|>\lambda\})]^{1/p}.$$

 Let $X$ and $Y$ be two Banach spaces. Given a bounded operator $T:X\rightarrow Y$, we denote the operator norm by $\|T\|_{X\rightarrow Y}$ which is defined in the standard way i.e. $\sup\limits_{\|f\|_{X}\le1}\|Tf\|_{Y}$. If $X=Y$ we use the symbol $\|T\|_{X}$.

A non-negative locally integrable function $w$ define on $\Bbb{R}^{n}$ is said to satisfy $A_p({\Bbb{R}}^{n})$ condition ($w\in A_p({\Bbb{R}}^{n})$) for $1<p<\infty$ if
$$\|w\|_{A_{p}({\Bbb{R}^{n}})}:=\sup\limits_{Q}\bigg(\frac{1}{|Q|}\int\limits_{Q}w(x)dx\bigg)\bigg(\frac{1}{|Q|}\int\limits_{Q}w(x)^{1-p'}dx\bigg)^{p-1}<\infty,$$
where $p'=\frac{p}{p-1}$ and supremum is taken over all cubes $Q$ in $\Bbb{R}^{n}$ with sides parallel to the co-ordinate axes. We call $\|w\|_{A_p({\Bbb{R}}^n)}$ the $A_p$ characteristic of $w$.

In 1972 B. Muckhenhoupt \cite{MUCK1} showed that if $w\in A_p(\Bbb{R}^n)$, where $1<p<\infty$,  then the Hardy--Littlewood maximal operator
$$Mf(x)=\sup\limits_{x\in Q}\frac{1}{|Q|}\int\limits_{Q}|f(y)|dy$$
is bounded in $L^p_w({\Bbb{R}}^n)$. Later R. Hunt, B. Muckhenhoupt and R. L. Wheeden \cite{HUNT-MUCK} proved  that the Hilbert transform
$$\mathcal{H}f(x)=p.v\int\limits_{\Bbb{R}}\frac{f(y)}{x-y}dy, \;\;\; x\in {\Bbb{R}}, $$
is also bounded in $L^p_w({\Bbb{R}})$ if $w\in A_p({\Bbb{R}})$.

S. Buckley \cite{Buck} investigated the sharp $A_p$ bound for the operator $M$ and established the inequality
\begin{equation}\label{f1}
\|M\|_{L^{p}_w({\Bbb{R}}^n)}\leq C\|w\|_{A_p({\Bbb{R}}^n)}^{\frac{1}{p-1}},\;\;\;\; 1<p<\infty,
\end{equation}
Moreover, he showed that the power $\frac{1}{p-1}$ is best possible in the sense that we can not replace $\|w\|_{A_p}^{\frac{1}{p-1}}$ by
$\psi(\|w\|_{A_p})$ for any positive non-decreasing function $\psi$ growing slowly than $x^{\frac{1}{p-1}}$. From here it follows that for any $\lambda>0$,
$$\sup\limits_{w\in A_p}\frac{\|M\|_{L^{p}_w}}{\|w\|_{A_p}^{\frac{1}{p-1}-\lambda}}=\infty.$$
 It was also shown by S. Buckley that for $1<p<\infty$ convolution Calder\'{o}n Zygmund singular operator satisfies
 $$\|T\|_{L^p_w(\Bbb{R}^n)}\le c\|w\|_{A_p(\Bbb{R}^n)}^{\frac{p}{p-1}}$$
and the best possible exponent is at least $\max\{1,\frac{1}{p-1}\}$. S. Petermichl \cite{PET1}, \cite{PET2} proved that the estimate
$$\|S\|_{L^p_w(\Bbb{R}^n)}\leq c\|w\|_{A_p(\Bbb{R}^n)}^{\max\{1,\frac{1}{p-1}\}}$$
is sharp, where $S$ is either the Hilbert transform or one of the Riesz transforms in $\Bbb{R}^n$
$$R_jf(x)=c_{n} p.v\int\limits_{{\Bbb{R}}^{n}}\frac{x_{j}-y_j}{|x-y|^{n+1}}f(y)dy.$$
S. Petermichl obtained the results for $p=2$. The general case $p\neq2$ then follows by the sharp version of he Rubio de Francia extrapolation theorem given by O. Dragi\v{c}evi\'{c}, L. Grafakos, C. Pereyra and S. Petermichl \cite{DRAG-GRAF-PR} (see also, T. Hyt\"onen \cite{HYT} regarding the $A_2$ conjecture for Calder\'{o}n-Zygmund operators which, in fact,implies appropriate estimate for all exponents $1<p<\infty$ by applying a sharp version of Rubio de Francia's extrapolation theorem.).

In 1974 B. Muckenhoupt and R. Wheeden \cite{MuWh} found necessary and sufficient condition for the one-weight inequality; namely, they proved that the Riesz potential $I_{\alpha}$ (resp the fractional maximal operator $M_{\alpha}$) is bounded from $L^p_{w^p}({\Bbb{R}}^n)$ to $L^q_{w^q}({\Bbb{R}}^n)$, where $1<p<\infty$, $0<\alpha<n/p$, $q= \frac{np}{n-\alpha p}$ if and only if $w$ satisfies  the so called $A_{p,q}({\Bbb{R}}^n)$ condition (see the definition below).  Moreover, from their result it follows that there is a positive constant $c$ depending only on $p$ and $\alpha$ such that
\begin{equation}\label{MW}
\| T_{\alpha} \|_{L^p_{w^p} \to L^q_{w^q}} \leq c \| w\|_{A_{p,q}}^{\beta},
\end{equation}
for some positive exponent  $\beta$, where $T_{\alpha}$ is either $I_{\alpha}$ or $M_{\alpha}$, and $\| w\|_{A_{p,q}} $ is the $A_{p,q}$ characteristic of $w$. M. T. Lacey, K. Moen, C. Perez and R. H. Torres \cite{LAC-MOE-PER-TORR}  proved that the best possible value of $\beta$ in \eqref{MW} is $p'/q(1-\alpha/n)$ (resp. $(1-\alpha/n)\max \{ 1, p'/q \}$)  for $M_{\alpha}$ (resp. for $I_{\alpha}$) (see also \cite{CrMo} for this and other sharp results).

In 1986 E. Sawyer proved the following inequality for the right  maximal operator $M^+$:
\begin{eqnarray} \label{H-L-sharp+}
\| M^+ f\|_{L^p_w({\Bbb{R}})} \leq C_{p} \| w\|_{A_{p}^+({\Bbb{R}})}^{\beta}  \| f\|_{L^p_w({\Bbb{R}})}, \;\;\; f\in L^p_{w}({\Bbb{R}}),
 \end{eqnarray}
 with some positive exponent $\beta$, where $\| w\|_{A_{p}^+({\Bbb{R}})}$ is  $A_p^+$ characteristic of a weight $w$ (see the appropriate definitions below).

Later K. Andersen and E. Sawyer,  in their celebrated work \cite{AND-SWA} completely characterized  the one-weight boundedness for one--sided fractional operators. In particular, they proved that if $1<p<\infty$, $0<\alpha<1/p$, $q=\frac{p}{1-\alpha p}$, then

\begin{eqnarray} \label{sharp+}
\| w {\mathcal{N}}^{+}_{\alpha} f\|_{L^q({\Bbb{R}})} \leq C_{p,\alpha} \| w\|_{A_{p,q}^+({\Bbb{R}})}^{\beta}  \| w f\|_{L^p({\Bbb{R}})}, \;\;\; f\in L^p_{w^p}({\Bbb{R}}),
 \end{eqnarray}
for some positive $\beta$, where ${\mathcal{N}}^{+}_{\alpha}$ is either the Weyl transform ${\mathcal{W}}_{\alpha}$ or the right fractional maximal operator $M^+_{\alpha}$, and $\| w\|_{A_{p,q}^+}$ is the right $A_{p,q}^+$ characteristic of a weight $w$ (see the appropriate definitions below).

Our aim is to  obtain sharp bounds  in inequalities \eqref{H-L-sharp+}, \eqref{sharp+} (as well as  to investigate their ''left'' analogs), in particular, to establish best possible value for $\beta$. From the obtained results for one--sided potentials, for example, can be easily obtained the sharp estimates for two--sided fractional integrals $I_{\alpha}$ in the case $n=1$ established in \cite{LAC-MOE-PER-TORR}. Known results and the  derived statements are  applied  to give analogous sharp estimates for multiple operators (strong maximal functions, multiple singular integrals and potentials with product kernels). In some cases, appropriate examples of weighted bounds are given.

To explain better the point of sharp estimates for multiple operators, let us discuss, for example,  the strong Hardy--Littlewood maximal operator $M^{(s)}$ defined on  ${\Bbb{R}}^2$. Denote by $A_p^{(s)}({\Bbb{R}}^2)$ the Muckenhoupt class taken with respect to the rectangles  with sides parallel to the co-ordinate axes (see Section \ref{Strong} for the definitions). Let $\|w\|_{A_p^{(s)}({\Bbb{R}}^2)}$ be $A_p^{(s)}$ characteristic of $w$.  There arises a natural questions regarding the sharp bound in the inequality
\begin{equation}\label{Rectangle}
\|M^{(s)} \|_{L^p_w({\Bbb{R}}^2)} \leq c \|w\|_{A_p^{(s)}({\Bbb{R}}^2)}^{\beta}.
\end{equation}
We show that the following estimate is sharp
\begin{equation}\label{variable}
\|M^{(s)}\|_{L^{p}_{w}({\Bbb{R}}^2)}\le c\Bigg( \|w\|_{A_p(x_{1})} \|w\|_{A_p(x_{2})}\Bigg)^{1/(p-1)},
\end{equation}
where $\|w\|_{A_p(x_{i})}$ be the characteristic of the weight $w$ defined with respect to the $i$- th variable  uniformly to another one $i=1,2$ (see e.g., \cite{FeSt}, \cite{KOK}, \cite{GARC}, Ch. IV for the one--weight theory for multiple integral operators). Inequality \eqref{variable} together with the Lebesgue differentiation theorem implies that \eqref{Rectangle} holds for $\beta= \frac{n}{p-1}$; however, unfortunately we do not know whether it is or not sharp.

The proofs of the main results are based on the two-weight theory for one--sided fractional integrals and the ideas of \cite{Buck};  \cite{LAC-MOE-PER-TORR}; \cite{PET1}, \cite{PET2}.

Under the symbol $A\approx B$ we mean that there are positive  constants $c_1$ and $c_2$  (depending on appropriate parameters) such that $ c_1A \leq B \leq c_2 A$;  $A \ll B$ means that there is a positive constant $c$ such that $A \leq  c B$

Finally we mention that constants (often different constants in one and the same lines of inequalities) will be denoted by $c$ or $C$. The symbol $p'$ stands for  the conjugate number of $p$:  $p'=p/(p-1)$, where $1<p<\infty$.

\section{One-sided Hardy--Littlewood maximal operator}
Let $f$ be a locally integrable function on $\Bbb{R}$. Then we define one--sided Hardy--Littlewood maximal functions as:

$$M^{+}f(x)=\sup \limits_{h>0} \frac{1}{ h } \int\limits_{x}^{x+h} |f(t)|dt, \;\;\;\;  M^{-}f(x)=\sup\limits_{h>0}\frac{1}{h}\int\limits_{x-h}^h |f(t)|dt, \;\;\;\;\; x\in\Bbb{R}.$$

E. Sawyer \cite{SWA} characterized the one-weighted inequality for $M^{+}$ and $M^{-}$ under the so-called one--sided Muckenhoupt condition; namely he proved that if $1<p<\infty$, then

(i) the inequality

$$ \| M^{+} f\|_{L^p_w({\Bbb{R}})} \leq C \| f\|_{L^p_w({\Bbb{R}})}, \;\;\; f\in L^p_w({\Bbb{R}})$$
holds if and only if $w\in A_p^{+}(\Bbb{R})$, i.e.,

\begin{equation*}
\|w\|_{A_{p}^{+}(\Bbb{R})}:=\sup\limits_{x\in\Bbb{R},h>0}\bigg(\frac{1}{h}\int\limits_{x-h}^{x}w(t)dt\bigg)\bigg(\frac{1}{h}\int\limits^{x+h}_{x}
w^{1-p'}(t)dt\bigg)^{p-1}<\infty.
\end{equation*}

(ii) the inequality

$$ \| M^{-} f\|_{L^p_w({\Bbb{R}})} \leq C \| f\|_{L^p_w({\Bbb{R}})}, \;\;\; f\in L^p_w({\Bbb{R}})$$
holds if and only if $w\in A_p^{-}(\Bbb{R})$, i.e.,

\begin{equation*}
\|w\|_{A_{p}^{-}(\Bbb{R})}:= \sup\limits_{x\in\Bbb{R},h>0}\bigg(\frac{1}{h}\int\limits_{x}^{x+h}w(t)dt\bigg)\bigg(\frac{1}{h}\int\limits^{x}_{x-h}
w^{1-p'}(t)dt\bigg)^{p-1}<\infty,
\end{equation*}
\vskip+0.2cm
Our main result of this section reads as follows:

\begin{thm}\label{m8}
Let $1<p<\infty$. Then

\rm{(i)}
$$\|M^{+}\|_{L^p_w(\Bbb{R})}\leq c\|w\|_{A_{p}^{+}(\Bbb{R})}^{\frac{1}{p-1}}$$
holds and the exponent $\frac{1}{p-1}$ is best possible.

\rm{(ii)}
$$\|M^{-}\|_{L^p_w(\Bbb{R})}\leq c\|w\|_{A_{p}^{-}(\Bbb{R})}^{\frac{1}{p-1}}$$
holds and the exponent $\frac{1}{p-1}$ is best possible.
\end{thm}
To prove Theorem \ref{m8} we need some auxiliary statements. The fact that $A_p^+$ (resp $A_p^-$) condition implies $A_{p-\varepsilon}^+$ (resp. $A_{p-\varepsilon}^-$) is well-known (see e.g., the papers \cite{MRe} and \cite{SWA}) but we formulate and prove the next lemma because of sharp estimates.
\begin{lema}\label{z1}
If $w\in A_p^{+}(\Bbb{R})$ $($resp. $w\in A_p^{-}(\Bbb{R})$$)$, then $w\in A_{p-\epsilon}^+(\Bbb{R})$ $($resp. $w\in A_{p-\epsilon}^{-}(\Bbb{R})$$)$, where $\epsilon\approx \|w\|^{1-p'}_{A_p^+(\Bbb{R})}$ $($resp. $\epsilon\approx \|w\|^{1-p'}_{A_p^-(\Bbb{R})}$ $)$ and
$\|w\|_{A_{p-\epsilon}^{+}(\Bbb{R})}\le c\|w\|_{A_p^+(\Bbb{R})}$, $($ resp. $\|w\|_{A_{p-\epsilon}^{-}(\Bbb{R})}\le c\|w\|_{A_p^-(\Bbb{R})}$ $)$.
\end{lema}
We only give the proof of this Lemma for $w\in A_p^+(\Bbb{R})$ the case $w\in A_p^{-}(\Bbb{R}) $ follows analogously.

\vskip+0.2cm

It should be noted that if $w\in A_p^{+}(\Bbb{R})$ then $w$ in general does not satisfy reverse H\"{o}lder inequality (see \cite{SWA}), but nevertheless it satisfies the so called weak reverse H\"{o}lder inequality which is stated as follows:
\begin{lema}[\cite{MRe}]\label{z3}
Let $1<p<\infty$. If $w$ satisfies $A_p^+(\Bbb{R})$ then there exist  positive constants $C$ and $c$ independent of $w$, $a$  and $b$ such that
$$\int_{a}^{b}w^{1+\delta}(x) dx \le C (M^{-}(w\chi_{(a,b)}(b)))^\delta \int_a^b w(x) dx, $$
where $\delta=\frac{1}{c\|w\|_{A_p^{+}(\Bbb{R})}}$,  and therefore
$$M^{-}(w^{1+\delta}\chi_{(a,b)})(b)\le C (M^{-}(w_{\chi_{(a,b)}})(b))^{1+\delta}$$
with the same constants $C$ and $c$.
\end{lema}
\begin{proof}
The proof of this lemma is given in \cite{MRe}, but here we present the proof only for estimate of $\delta$. Following the proof of Lemma 5 in that paper we have the following estimate
$$\bigg(1-\frac{\delta}{(1+\delta)\alpha\beta^{1+\delta}}\bigg)\int\limits_a^b w^{1+\delta}(x) dx \le (M^{-}(w\chi_{(a,b)}(b)))^\delta \int\limits_a^b w(x) dx ,$$
where $\alpha= 1- 2 \big( 4 \|w\|_{A_p^+({\Bbb{R}})}\beta)^{1/(p-1)}$ and  $\beta$ can be taken appropriately later.   Finally, by assuming that $\beta = 1/ (2^{2p}  \|w\|_{A_p^+({\Bbb{R}})})$ (in this case $\alpha=1/2$) we have that
$$\frac{\delta}{(1+\delta)\alpha\beta^{1+\delta}} \leq 1 $$
taking $\delta=\frac{1}{c\|w\|_{A_p^+(\Bbb{R})}}$  with the constant $c$ large enough independently $w$, $a$ and $b$.
The lemma is proved.
\end{proof}
The following Theorem is in \cite{SWA}.
\begin{thma}\label{z7}
Suppose that $1<p<\infty$. Then

\rm{(i)} there is a positive constant $c_p$ depending only on $p$ such that the weak type inequalities
$$w(\{M^{+}f>\alpha\})\le c_p \|w\|_{A_p^{+}(\Bbb{R})}\bigg(\frac{\|f\|_{L^p_w(\Bbb{R})}}{\alpha}\bigg)^p$$
\rm{(ii)}
$$w(\{M^{-}f>\alpha\})\le c_p \|w\|_{A_p^{-}(\Bbb{R})}\bigg(\frac{\|f\|_{L^p_w(\Bbb{R})}}{\alpha}\bigg)^p$$
holds.
\end{thma}
Now we recall the Marcinkiewicz interpolation theorem.
\begin{thma}\label{z9}
Suppose $1\le p_0<p_1<\infty$ and that $T$ is a sublinear operator of weak-type $p_0$ and $p_1$, with respect to the measure $d\mu=wdx$, with norms $R_0$ and $R_1$ respectively, then $T$ is  bounded on $L^p_w(\Bbb{R})$ for all $p_0<p<p_1$. In fact, for any $0<t<1$,
$$\|Tf\|_{L^{p_t}_w(\Bbb{R})}\le C_tR_0^{1-t}R_1^{t}\|f\|_{L^{p_t}_w(\Bbb{R})}$$
where
$$\frac{1}{p_t}=\frac{1-t}{p_0}+\frac{t}{p_1}\indent \textrm{and} \;\;C_t^p=\frac{2^{p_t}}{p_t}\bigg(\frac{p_1}{p_1-p_t}+
\frac{p_0}{p_t-p_0}\bigg).$$
\end{thma}

\vskip+0.2cm

{\emph Proof of Lemma \ref{z1}}. We follow the proof of Proposition 3 in \cite{MRe}. For simplicity let us denote $s=p-\epsilon$. To show $w\in A_{s}^{+}(\Bbb{R})$ , $1<s<p$, it is enough to show that
$${\mathcal{M}}_s^+(w) := \sup\limits_{a,b}\sup\limits_{a<x<b}\frac{1}{(b-a)^s}\bigg( \int\limits_{a}^{x}w(t) dt \bigg) \bigg(\int\limits_{x}^{b} w^{-1/(s-1)(t) dt}\bigg)^{s-1}\ll \|w\|_{A_+^p({\Bbb{R}})} $$
because (see \cite{MRe})
$$ {\mathcal{M}}_r^+(w) \approx \|w\|_{A_r^+({\Bbb{R}})}, \;\;\;\; 1<r<\infty. $$

Observe that $w\in A_{p}^{+}(\Bbb{R})$ if and only if $\sigma\in A_{p'}^{-}(\Bbb{R})$. Then by analogue Lemma \ref{z3} for $A_{p'}^{-}(\Bbb{R})$ classes, we have that there is a constant $C$ independent of $w$, $x$ and $b$  such that
\begin{equation}\label{z4}
M^{+}(\sigma^{1+\delta}\chi_{(x,b)})(x)\le C (M^{+}(\sigma_{\chi_{(x,b)}})(x))^{1+\delta},
\end{equation}
 where $\delta=\frac{1}{c\|\sigma\|_{A_{p'}^{-}(\Bbb{R})}}=\frac{1}{c\|w\|_{A_{p}^{+}(\Bbb{R})}^{p'-1}}$. Now we claim that $w\in A_s^{+}({\Bbb{R}})$, where $s=\frac{p+\delta}{1+\delta}$. Fix $a<x<b$. Since $\sigma$ is locally integrable, it follows from \eqref{z4} that same holds for $\sigma^{1+\delta}$. Hence, there exists a finite deceasing sequence $x_0=x>x_1>\cdots>x_{N}\ge a=x_{N+1}$ such that
\begin{equation}\label{z5}
\int\limits_{a}^{b}\sigma^{1+\delta}=2^k\int\limits_{x}^{b}\sigma^{1+\delta}\indent \textrm{if}\;\; k=0,\cdots,N \;\textrm{and}\;
\int\limits_{a}^{x_N}\sigma^{1+\delta}<2^N\int\limits_{x}^{b}\sigma^{1+\delta}.
\end{equation}
From \eqref{z5} it  follows the fact that for every $k=0,\cdots,N$
\begin{equation}\label{z6}
\int\limits_{x_{k+1}}^{b}\sigma^{1+\delta}\le2^{k+1}\int\limits_{x}^{b}\sigma^{1+\delta}.
\end{equation}
On the other hand, \eqref{z4} and \eqref{z5} yield
\begin{eqnarray*}
\int\limits_a^x w\bigg(\frac{1}{b-a}\int\limits_{x}^{b}\sigma^{1+\delta}\bigg)^s&=&\sum\limits_{k=0}^{N}\frac{1}{2^{ks}}\int\limits_{x_{k+1}}^{x_k}
w\bigg(\frac{1}{b-a}\int\limits_{x_k}^{b}\sigma^{1+\delta}\bigg)^s\\
&\le&\sum\limits_{k=0}^{N}\frac{1}{2^{ks}}\int\limits_{x_{k+1}}^{x_k}
w(y)(M^{+}(\sigma^{1+\delta}\chi_{(y,b)})(y))^s dy\\
&\le&\sum\limits_{k=0}^{N}\frac{2}{2^{ks}}\int\limits_{x_{k+1}}^{x_k}
w(y)(M^{+}(\sigma \chi_{(x_{k+1},b)})(y))^{p+\delta} dy.
\end{eqnarray*}
Since $w\in A_p^{+}(\Bbb{R})$, we know by Theorem \ref{z7} that $M^{+}$ maps  $L^{p}_{w}(\Bbb{R})$ into weak $L^{p}_{w}(\Bbb{R})$ with the norm less than or equal to $\|w\|_{A_p^{+}(\Bbb{R})}^{1/p}$. Then, by the Marcinkiewicz interpolation theorem, $M^+$ maps $L^{p+\delta}_{w}(\Bbb{R})$
into $L^{p+\delta}_{w}(\Bbb{R})$ with the norm $\leq \|w\|^{1/(p+\delta)}_{A_p^+({\Bbb{R}})}$. This together with \eqref{z6} and the fact that $s>1$ gives
\begin{eqnarray*}
\int\limits_a^x w\bigg(\frac{1}{b-a}\int\limits_{x}^{b}\sigma^{1+\delta}\bigg)^s&\le&2\|w\|_{A_{p}^{+}(\Bbb{R})}\sum\limits_{k=0}^{N}\frac{1}{2^{ks}}\int\limits_{x_{k+1}}^{b}
\sigma^{1+\delta}\\
&\le&2\|w\|_{A_{p}^{+}(\Bbb{R})}\sum\limits_{k=0}^{N}\frac{2^{k+1}}{2^{ks}}\int\limits_{x}^{b}
\sigma^{1+\delta}\le 2\|w\|_{A_{p}^{+}}\int\limits_{x}^{b}\sigma^{1+\delta}<\infty.
\end{eqnarray*}
Now the lemma follows. Further, notice that $\epsilon=\frac{p-1}{1+c\|w\|^{p'-1}_{A_p^{+}(\Bbb{R})}}\approx c_p\|w\|^{1-p'}_{A_{p}^{+}(\Bbb{R})}$. $\;\;\;\; \Box$
\vskip+0.2cm

{\em Proof of Theorem \ref{m8}}. (i) We follow \cite{Buck}. Let $w\in A_{p}^{+}(\Bbb{R})$. Then by Lemma \ref{z1}, $w\in A_{p-\epsilon}^{+}(\Bbb{R})$, and  trivially, $w$ is also an $A_{p+\epsilon}^{+}(\Bbb{R})$ weight. Moreover, the following inequalities
$$\|w\|_{A_{p-\epsilon}^{+}(\Bbb{R})}\le c\|w\|_{A_p^+(\Bbb{R})},$$
$$ \|w\|_{A_{p+\epsilon}^{+}(\Bbb{R})}\le \|w\|_{A_p^+(\Bbb{R})}.$$
hold. Applying the  Marcinkiewicz interpolation theorem  corresponding to weak type results at $p-\epsilon$ and $p+\epsilon$, we get  estimate
$$\|M^{+}f\|_{L^{p}_{w}(\Bbb{R})}\le c\|w\|_{A_{p}^{+}(\Bbb{R})}^{\frac{1}{p-1}}\|f\|_{L^{p}_{w}({\Bbb{R}})}.$$
{\em Sharpness.} Let us take $0<\epsilon<1$. Let $w(x)=|1-x|^{(1-\epsilon)(p-1)}$. Then it is easy to check that
$$\|w\|_{A_{p}^{+}(\Bbb{R})}^{1/(p-1)}\approx\frac{1}{\epsilon}.$$
Observe also that for
$$f(t)=(1-t)^{\epsilon(p-1)-1}\chi_{(0,1)}(t),$$
we have $f\in{L^p_w}$. Now let $0<x<1$. Then we find that the following estimate
$$M^{+}f(x)\ge\frac{1}{1-x}\int\limits_{x}^{1}f(t)dt=c\frac{1}{\epsilon}f(x)$$
holds. Finally we conclude
$$\|M^{+}f\|_{L^{p}_w(\Bbb{R})}\ge c\frac{1}{\epsilon}\|f\|_{L^{p}_{w}({\Bbb{R}})}.$$
Thus the exponent $1/(p-1)$ is sharp.

Proof of (ii) is similar. We show only {\em Sharpness.} For that let us take $0<\epsilon<1$. Suppose that $w(x)=|x|^{(1-\epsilon)(p-1)}$. Then it is easy to check that $\|w\|_{A_{p}^{-}(\Bbb{R})}^{1/(p-1)}\approx\frac{1}{\epsilon}.$ Observe also that for $f(t)=t^{\epsilon(p-1)-1}\chi_{(0,1)}(t)$, $f\in{L^p_w}({\Bbb{R}})$.

Now let $0<x<1$. Then we have the following estimate
$$M^{-}f(x)\ge\frac{1}{x}\int\limits_{0}^{x}f(t)dt=c\frac{1}{\epsilon}f(x).$$
Finally
$$\|M^{-}f\|_{L^{p}_w(\Bbb{R})}\ge c\frac{1}{\epsilon}\|f\|_{L^{p}_{w}({\Bbb{R}})}.$$
Thus the exponent $1/(p-1)$ is sharp. $\;\;\;\; \Box$
\section{One-sided fractional integrals}

Let $f$ be a locally integrable function on $\Bbb{R}$ and let $0<\alpha<1$. Then we define one--sided fractional  integrals

$$\mathcal{W}_{\alpha}f(x)=\int\limits_{x}^{\infty}\frac{f(t)}{(t-x)^{1-\alpha}}dt,\indent\indent \mathcal{R}_{\alpha}f(x)=\int\limits_{-\infty}^{x}\frac{f(t)}{(x-t)^{1-\alpha}}dt\indent\indent x\in {\Bbb{R}},$$
and corresponding fractional maximal functions:

$$M_{\alpha}^{+}f(x)=\sup \limits_{h>0} \frac{1}{ h^{1-\alpha} } \int\limits_{x}^{x+h} |f(t)| dt, \;\;\;\;  M_{\alpha}^{-}f(x)=\sup\limits_{h>0}\frac{1}{h^{1-\alpha}}\int\limits_{x-h}^h |f(t)|dt, \;\;\;\;\; x\in\Bbb{R}.$$

Taking formally  $\alpha=0$, we have one--sided Hardy--Littlewood maximal operators.

%
%
%
%
%
%

As it was mentioned above (see Introduction)  necessary and sufficient conditions governing the one-weight inequality for one--sided fractional integrals under the $A_{p,q}^{\pm}$ conditions were established in \cite{AND-SWA}. Now we give the definition of $A_{p,q}^{\pm}$ classes:

\begin{defin}
Let $1<p,q<\infty$. We say that a weight function $w$ defined on $\Bbb{R}$ satisfies the $A_{p,q}^{+}(\Bbb{R})$ condition $(w\in A_{p,q}^{+}(\Bbb{R}))$, if
\begin{equation*}
\|w\|_{A_{p,q}^{+}(\Bbb{R})}:=\sup\limits_{\substack{x\in\Bbb{R}\\h>0}}\bigg(\frac{1}{h}\int\limits_{x-h}^{x}w^q(t)dt\bigg)\bigg(\frac{1}{h}\int\limits^{x+h}_{x}
w^{-p'}(t)dt\bigg)^{q/p'}<\infty;
\end{equation*}
a weight function $w$ satisfies the $A_{p,q}^{-}(\Bbb{R})$ condition $(w\in A_{p,q}^{-}(\Bbb{R}))$, if
\begin{equation*}
\|w\|_{A_{p,q}^{-}(\Bbb{R})}:=\sup\limits_{\substack{x\in\Bbb{R}\\h>0}}\bigg(\frac{1}{h}\int\limits_{x}^{x+h}w^q(t)dt\bigg)\bigg(\frac{1}{h}\int\limits^{x}_{x-h}
w^{-p'}(t)dt\bigg)^{q/p'}<\infty.
\end{equation*}
\end{defin}

\subsection{The two-weight problem}

It is well-known two-weight criteria of various type for one-sided fractional integrals. The Sawyer type two-weight characterization  for the operators $M^+_{\alpha}$, $M^-_{\alpha}$ reads as follows (see \cite{MaTo}):

\begin{thma}\label{Sa1}
Let $1<p\leq q<\infty$ and let $0<\alpha<1$. Suppose that $v$ and $w$ be weight functions on $\Bbb{R}$. Then\\
\rm{(i)}  $M^+_{\alpha}$ is bounded from $L^p_w(\Bbb{R})$ to $L^{q}_v(\Bbb{R})$ if and only if $(v,w)\in S^+_{p,q}$, i.e.

$$A^+_{MT}(v,w, p,q):= \sup_{I} \frac{\Big\| M^{+}_{\alpha}(\sigma \chi_I)\|_{L^q_v({\Bbb{R}})}}{\big(\sigma(I)\big)^{1/p}}<\infty, $$
where the supremum is taken over all intervals $I$  with $\sigma(I):= w^{1-p'}(I)<\infty$.
Moreover, $\|M^+_{\alpha}\|_{L^p_w \to L^{q}_v}\approx A^+_{MT}(v,w,p,q)$.

\noindent\rm{(ii)}  $M^-_{\alpha}$ is bounded from $L^p_w(\Bbb{R})$ to $L^{q}_v(\Bbb{R})$ if and only if $(v,w)\in S^-_{p,q}$, i.e.

$$A^-_{MT}(v,w, p,q):= \sup_{I} \frac{\Big\| M^{-}_{\alpha}(\sigma \chi_I)\|_{L^q_v({\Bbb{R}})}}{\big(\sigma(I)\big)^{1/p}}<\infty, $$
where the supremum is taken over all intervals $I$  with $\sigma(I):= w^{1-p'}(I)<\infty$.
Moreover, $\|M^-_{\alpha}\|_{L^p_w \to L^{q}_v}\approx A^-_{MT}(v,w,p,q)$.

\end{thma}

M. Lorente and A. De la Torre (\cite{LOR2}) gave necessary and sufficient conditions for the boundness of one--sided potentials from $L^p_w(\Bbb{R})$ to $L^q_v(\Bbb{R})$ (resp. to $L^{q, \infty}(\Bbb{R})$) under the Sawyer type conditions.

\vskip+0.1cm

In what follows, ${\mathcal{N}}_{\alpha}$ is one of the operators  ${\mathcal{W}}_{\alpha}$ or ${\mathcal{R}}_{\alpha}$. Observe that if ${\mathcal{N}}^*_{\alpha}$ denotes the adjoint of ${\mathcal{N}}_{\alpha}$, then it is easy to see that ${\mathcal{N}}^*_{\alpha}= {\mathcal{R}}_{\alpha}$ if ${\mathcal{N}}_{\alpha}= {\mathcal{W}}_{\alpha}$, and viceversa.

\begin{thma}\label{Sa2}
Let $1<p\leq q<\infty$ and let $0<\alpha<1$. Suppose that $v$ and $w$ be weight functions on $\Bbb{R}$.
Then

\rm{(i)}  ${\mathcal{N}}_{\alpha}$ is bounded from $L^p_w({\Bbb{R}})$ to $L^{q, \infty}_v({\Bbb{R}})$  if and only if  there is a positive constant $C$ such that
$$ A^*_{LT}(v,w, p,q):= \sup_{I} \frac{ \Big \| {\mathcal{N}}^*_{\alpha}(v\chi_I) \Big\|_{ L^{p'}_{ w^{1-p'} } ({\Bbb{R}})} } {\big(v(I)\big)^{1/q'}}<\infty, $$
where, the supremum is taken over all intervals $I$ with $v(I)<\infty$.
Moreover, $\|{\mathcal{N}}_{\alpha}\|_{L^p_w \to L^{q, \infty}_v} \approx A^*_{LT}(v,w,p,q)$.

\rm{(ii)}  ${\mathcal{N}}_{\alpha}$ is bounded from $L^p_w({\Bbb{R}})$ to $L^{q}_v({\Bbb{R}})$ if and only if

$$\sup_{I} \frac{\| {\mathcal{N}}_{\alpha}(\sigma \chi_I)\|_{L^q_v({\Bbb{R}})}} {\big(\sigma(I)\big)^{1/p}}<\infty, $$
where the supremum is taken over all intervals $I$ with $\sigma(I)<\infty$;

$$\sup_{I} \frac{\| {\mathcal{N}}^*_{\alpha}(\sigma \chi_I)\|_{ L^{p'}_\sigma ({\Bbb{R}})}}{\big(v(I)\big)^{1/p}}<\infty, $$
where $\sigma:=w^{1-p'}$ and $v(I)<\infty$.
\end{thma}

The following statements are from \cite{ED-MES-KOK} (see also $\S2.2$ of \cite{EdKoMe-Mon}) and give Gabidzashvili--Kokilashvili type two-weight criteria for one-sided potentials (we refer to  \cite{KOK-KR}, pp. 189-190 for the latter conditions in the case of Riesz potentials).


\begin{thma}\label{m21}
Let $1<p<q<\infty$ and let $0<\alpha<1$. Suppose that $v$ and $w$ be weight functions on $\Bbb{R}$. Then $\mathcal{W}_{\alpha}$ is bounded from $L^p_w({\Bbb{R}})$ to $L^{q,\infty}_v({\Bbb{R}})$ if and only if
$$[v,w]^{+}_{Glo}(p,q):=\sup\limits_{\substack{a\in\Bbb{R}\\h>0}}\Bigg(\int\limits_{a-h}^{a+h}v(t)dt\Bigg)^{1/q}\Bigg(\int\limits^{\infty}_{a+h}
(t-a)^{(\alpha-1)p'}w^{1-p'}(t)dt\Bigg)^{1/p'}<\infty.$$
Moreover, $\|\mathcal{W}_{\alpha}\|_{L^{p}_w\rightarrow L^{q,\infty}_v}\approx[v,w]_{Glo}^{+}(p,q)$.
\end{thma}

\begin{thma}\label{m20}
Let $1<p<q<\infty$ and let $0<\alpha<1$. Suppose that $v$ and $w$ be weight functions on $\Bbb{R}$. Then $\mathcal{R}_{\alpha}$ is bounded from $L^p_w({\Bbb{R}})$ to $L^{q,\infty}_v({\Bbb{R}})$ if and only if
$$[v,w]^{-}_{Glo}(p,q):=\sup\limits_{\substack{a\in\Bbb{R}\\h>0}}\Bigg(\int\limits_{a-h}^{a+h}v(t)dt\Bigg)^{1/q}\Bigg(\int\limits_{-\infty}^{a-h}
(a-t)^{(\alpha-1)p'}w^{1-p'}(t)dt\Bigg)^{1/p'}<\infty.$$
Moreover, $\|\mathcal{R}_{\alpha}\|_{L^{p}_w\rightarrow L^{q,\infty}_v}\approx[v,w]_{Glo}^{-}(p,q)$.
\end{thma}

\begin{thma}\label{m23}
Let $1<p<q<\infty$ and let $0<\alpha<1$. Suppose that $v$ and $w$ be weight functions on $\Bbb{R}$. Then $\mathcal{W}_{\alpha}$ is bounded from $L^p_w({\Bbb{R}})$ to $L^{q}_v({\Bbb{R}})$ if and only if\\
{\rm (i)} $[v,w]^{+}_{Glo}(p,q)<\infty$.\\
{\rm (ii)}\begin{equation*}\label{GK+}
A^{+}_{GK}(v,w,p,q):=\sup\limits_{\substack{a\in\Bbb{R}\\h>0}}
\Bigg(\int\limits_{a-h}^{a+h}w^{1-p'}(t)dt\Bigg)^{1/p'}\Bigg(\int\limits_{-\infty}^{a-h}
\frac{v(y)}{(a-y)^{(1-\alpha)q}}dy\Bigg)^{1/q}<\infty.
\end{equation*}
Moreover, $\|\mathcal{W}_{\alpha}\|_{L^{p}_w\rightarrow L^{q}_v}\approx[v,w]_{Glo}^{+}(p,q)+A^{+}_{GK}(v,w,p,q)$.
\end{thma}

\begin{thma}\label{m22}
Let $1<p<q<\infty$ and let $0<\alpha<1$. Suppose that $v$ and $w$ be weight functions on $\Bbb{R}$. Then $\mathcal{R}_{\alpha}$ is bounded from $L^p_w({\Bbb{R}})$ to $L^{q}_v({\Bbb{R}})$ if and only if\\
{\rm (i)} $[v,w]^{-}_{Glo}(p,q)<\infty$.\\

\noindent{\rm (ii)}
\begin{equation*}\label{GK-}
A^{-}_{GK}(v,w,p,q):=\sup\limits_{\substack{a\in\Bbb{R}\\h>0}}\Bigg(\int\limits_{a-h}^{a+h}w^{1-p'}(t)dt\Bigg)^{1/p'}\Bigg(\int\limits^{\infty}_{a+h}
\frac{v(y)}{(y-a)^{(1-\alpha)q}}dy\Bigg)^{1/q}<\infty.
\end{equation*}
Moreover, $\|\mathcal{R}_{\alpha}\|_{L^{p}_w\rightarrow L^{q}_v}\approx[v,w]_{Glo}^{-}(p,q)+ A^{-}_{GK}(v,w,p,q)$.
\end{thma}

\begin{rem} The strong type statements (see Theorems \ref{m23} and \ref{m22}) follow from appropriate weak type results (see Theorems \ref{m21} and \ref{m20}) and the Sawyer type two-weight criteria (see Theorem \ref{Sa2}). We refer to $\S2.2$ of \cite{EdKoMe-Mon} for details.
\end{rem}

\vskip+0.2cm

\begin{rem} It is easy to see that,\\
(i) \begin{equation}\label{m33}
\|\mathcal{R}_{\alpha}\|_{L^{p}_w\rightarrow L^{q}_v}\approx\|\mathcal{R}_{\alpha}\|_{L^{p}_w\rightarrow L^{q,\infty}_v}+\|\mathcal{W}_{\alpha}\|_{L^{q'}_{v^{1-q'}}\rightarrow L^{p',\infty}_{w^{1-p'}}};
\end{equation}
(ii)\begin{equation}\label{m34}
\|\mathcal{W}_{\alpha}\|_{L^{p}_w\rightarrow L^{q}_v}\approx\|\mathcal{W}_{\alpha}\|_{L^{p}_w\rightarrow L^{q,\infty}_v}+\|\mathcal{R}_{\alpha}\|_{L^{q'}_{v^{1-q'}}\rightarrow L^{p',\infty}_{w^{1-p'}}}.
\end{equation}
\end{rem}

We need similar criteria for one--sided fractional maximal operators defined on ${\Bbb{R}}$ (cf. \cite{Wh}).

\begin{prop}\label{GaKo1} Suppose that $1<p<q<\infty$ and that $0<\alpha<1$. Then

\rm{(i)} $M^+_{\alpha}$ is bounded from $L^p_w({\Bbb{R}})$ to $L^q_v({\Bbb{R}})$ if and only if \\$A^{+}_{GK}(v,w,p,q)<\infty$ (see Theorem \ref{m23}).

Moreover, $\|M_{\alpha}^+\|_{L^{p}_w\rightarrow L^{q}_v}\approx A^+_{GK}(v,w,p,q)$.

\rm{(ii)} $M^-_{\alpha}$ is bounded from $L^p_w({\Bbb{R}})$ to $L^q_v({\Bbb{R}})$ if and only if \\$A^{-}_{GK}(v,w,p,q)<\infty$ (see Theorem \ref{m22}).

Moreover, $ \| M^-_{\alpha} \|_{ L^{p}_w \rightarrow L^{q}_v }\approx A^-_{GK}(v,w,p,q)$.
\end{prop}

\begin{proof} We prove (i). Proof for (ii) is similar. Let $A^+_{GK}(v,w,p,q)<\infty$. Observe that
$$A^+_{GK}(v,w,p,q)= [w^{1-p'},v^{1-q'}]_{Glo}^{-}(q',p').$$
Hence, by Theorem \ref{m20} we have that
$$\|{\mathcal{R}}_{\alpha}\|_{ L^{q'}_{v^{1-q'}} \rightarrow L^{p',\infty}_{w^{1-p'}}} \approx A^+_{GK}(v,w,p,q).$$ Applying now Theorem \ref{Sa2} (part (i)) for ${\mathcal{N}}_{\alpha}= {\mathcal{R}}_{\alpha}$ (in that case ${\mathcal{N}}^*_{\alpha}= {\mathcal{W}}_{\alpha}$) we conclude that
$$A^*_{LT}(w^{1-p'}, v^{1-q'}, q',p') \ll \|\mathcal{R}_{\alpha}\|_{ L^{q'}_{v^{1-q'}} \rightarrow L^{p',\infty}_{w^{1-p'}}} \ll A^+_{GK}(v,w,p,q) < \infty.$$

Hence,  by Theorem \ref{Sa1} we have that

$$\| M^+_{\alpha}\|_{L^p_w \to L^q_v}\ll A^+_{MT}(v,w,p,q) \ll A^*_{LT}(w^{1-p'}, v^{1-q'}, q',p') \ll A^+_{GK}(v,w,p,q) <\infty.$$  Necessity, and consequently, the inequality
$$A^+_{GK}(v,w,p,q) \ll \|\mathcal{R}_{\alpha}\|_{L^{p}_w\rightarrow L^{q}_v}  $$
follows by the standard way choosing appropriate test functions.
\end{proof}

\subsection{The main results}

Now we are in the position to formulate our main results of this section.

\begin{thm}\label{m10}
Suppose that $0<\alpha<1$, $1<p<1/\alpha$ and that $q$ is such that $1/p-1/q-\alpha=0$. Then

\rm{(i)} there exists a positive constant $c$ depending only on $p$  and $\alpha$ such that

\begin{equation}\label{SharpMax}
\|M^{+}_{\alpha}\|_{L^p_{w^p}\rightarrow L^q_{w^q}}\leq c  \|w\|_{A_{p,q}^{+}(\Bbb{R})}^{\frac{p'}{q}(1-\alpha)}.
 \end{equation}
 Moreover, the exponent ${\frac{p'}{q}(1-\alpha)}$ is best possible.

\rm{(ii)}  there exists a positive constant $c$ depending only on $p$ and $\alpha$ such that

\begin{equation}\label{SharpMax1}
\| M^{-}_{\alpha} \|_{ L^p_{w^p}\rightarrow L^q_{w^q} } \leq c \| w \|_{ A_{p,q}^{-}(\Bbb{R}) }^{\frac{p'}{q}(1-\alpha)}.
\end{equation}
Moreover, the  exponent ${\frac{p'}{q}(1-\alpha)}$ is best possible.
\end{thm}

\begin{thm}\label{m30}
Let $1< p<\frac{1}{\alpha}$, where $0<\alpha<1$. We set $q=\frac{p}{1-\alpha p}$. Then \\
\noindent{\textnormal(a)}
\begin{equation}\label{m31c}
\|\mathcal{R}_{\alpha}\|_{L^{p}_{w^p}\rightarrow L^{q,\infty}_{w^q}}\le c\|w\|_{A_{p,q}^{-}(\Bbb{R})}^{1-\alpha},
\end{equation}
where the positive constant $c$ depends only on $p$ and $\alpha$. \\
\noindent{\textnormal(b)}
\begin{equation}\label{m31a}
\|\mathcal{W}_{\alpha}\|_{L^{p}_{w^p}\rightarrow L^{q,\infty}_{w^q}}\le c\|w\|_{A_{p,q}^{+}(\Bbb{R})}^{1-\alpha},
\end{equation}
where the positive constant $c$ depends only on $p$ and $\alpha$.
\end{thm}

\begin{thm}\label{m32}
Let $0<\alpha<1$, $1 < p<1/\alpha$ and let $q$ satisfy $q=\frac{p}{1-\alpha p}$. Then \\
{\textrm(a)}  there is a positive  constant $c$ depending only on $p$ and $\alpha$ such that
\begin{equation}\label{m31b}
\|\mathcal{R}_{\alpha}\|_{L^{p}_{w^p}\rightarrow L^{q}_{w^q}}\le c\|w\|_{A_{p,q}^{-}(\Bbb{R})}^{(1-\alpha)\max\{1,p'/q\}}.
\end{equation}
Furthermore, this estimate is sharp;\\
{\textrm(b)}  there is a positive  constant $c$ depending only on $p$ and $\alpha$ such that
\begin{equation}\label{m31}
\|\mathcal{W}_{\alpha}\|_{L^{p}_{w^p}\rightarrow L^{q}_{w^q}}\le c\|w\|_{A_{p,q}^{+}(\Bbb{R})}^{(1-\alpha)\max\{1,p'/q\}}.
\end{equation}
Moreover, this estimate is sharp.
\end{thm}

\subsection{Main lemmas}

In this subsection we prove some lemmas regarding the weights satisfying one--sided conditions.

\begin{lem}\label{m24}
Let $u\in A_{p,q}^{-}(\Bbb{R})$. Then for every $a\in\Bbb{R}$ and $h>0$,
$$\int\limits_{a}^{a+h}u^q(x)dx\leq\|u\|_{A_{p,q}^{-}(\Bbb{R})}\int\limits_{a-h}^{a}u^q(x)dx.$$
\end{lem}

Further, let $u\in A_{p,q}^{+}(\Bbb{R})$. Then for every $a\in\Bbb{R}$ and $h>0$,
$$\int\limits_{a-h}^{a}u^q(x)dx\leq\|u\|_{A_{p,q}^{+}\Bbb{(R)}}\int\limits_{a}^{a+h}u^q(x)dx.$$
\begin{proof}
Let us prove the first part. The second one follows analogously. Let $u\in A_{p,q}^{-}(\Bbb{R})$. Using H\"{o}lder's inequality we have
\begin{eqnarray*}
\int\limits_{a}^{a+h}u^{q}(x)dx&\leq&\Bigg(\int\limits_{a}^{a+h}u^q(x)dx\Bigg)h^{-q}\Bigg(\int\limits_{a-h}^{a}u^q(x)dx\Bigg)
\Bigg(\int\limits_{a-h}^{a}u^{-q'}(x)dx\Bigg)^{q/q'}\\
&\leq&h^{-1}\Bigg(\int\limits_{a}^{a+h}u^q(x)dx\Bigg)\Bigg(\int\limits_{a-h}^{a}u^q(x)dx\Bigg)\Bigg(\frac{1}{h}\int\limits_{a-h}^{a}
u^{-p'}(x)dx\Bigg)^{q/p'}\\
&\leq&\|u\|_{A_{p,q}^{-}(\Bbb{R})}\int\limits_{a-h}^{a}u^{q}(x)dx.
\end{eqnarray*}
\end{proof}


\begin{lem}\label{m26}
Let $u\in A_{p,q}^{-}(\Bbb{R})$. Then for all $a\in\Bbb{R}$ and $r>0$ we have,
$$\frac{\int\limits_{a-r}^{a}u^q(x)dx}{\int\limits_{a-2r}^{a}u^q(x)dx}
\leq\frac{\|u\|_{A_{p,q}^{-}\Bbb{(R)}}}{\|u\|_{A_{p,q}^{-}\Bbb{(R)}}+1}.$$
Further, let $u\in A_{p,q}^{+}(\Bbb{R})$. Then for all $a\in\Bbb{R}$ and $r>0$ we have,
$$\frac{\int\limits_{a}^{a+r}u^q(x)dx}{\int\limits_{a}^{a+2r}u^q(x)dx}\leq
\frac{\|u\|_{A_{p,q}^{+}(\Bbb{R})}}{\|u\|_{A_{p,q}^{+}(\Bbb{R})}+1}.$$
\end{lem}
\begin{proof}
Let $u\in A_{p,q}^{-}(\Bbb{R})$. By Lemma \ref{m24} we find that
\begin{eqnarray*}
\int\limits_{a-2r}^{a}u^{q}(x)dx&=&\int\limits_{a-2r}^{a-r}u^{q}(x)dx+\int\limits_{a-r}^{a}u^{q}(x)dx\\
&\ge&\Bigg(\frac{1}{\|u\|_{A_{p,q}^{-}(\Bbb{R})}}+1\Bigg)\int\limits_{a-r}^{a}u^q(x)dx.
\end{eqnarray*}

The remaining part of the lemma follows analogously.
\end{proof}

To prove Theorems \ref{m10} and \ref{m30} we need to prove also the following Lemma:

\begin{lem}\label{m45}
Let $0<\alpha<1$, $1<p<1/\alpha$ and let $q$ be such that $\frac{1}{q}=\frac{1}{p}-\alpha$. Then
\begin{equation}\label{m28}
[w^q, w^p]^{+}_{Glo}\ll c\|w\|^{1-\alpha}_{A_{p,q}^{+}(\Bbb{R})}.
\end{equation}
\begin{equation}\label{m29}
[w^q, w^p]^{-}_{Glo}\ll c\|w\|^{1-\alpha}_{A_{p,q}^{-}(\Bbb{R})}.
\end{equation}
\end{lem}
\begin{proof}

We prove \eqref{m29}. Proof for \eqref{m28} follows in the same manner.  First observe that $p'(1-\alpha)=1+\frac{p'}{q}$.
Further, let us take $a\in\Bbb{R}$ and $h>0$. Then by Lemma \ref{m26},
\begin{eqnarray*}
&&\Bigg(\int\limits_{a-2h}^{a}w^{q}(x)dx\Bigg)^{1/q}
\Bigg(\int\limits_{-\infty}^{a-2h}(a-x-h)^{(\alpha-1)p'}w^{-p'}(x)dx\Bigg)^{1/p'}\\
&\le&c\Bigg(\int\limits_{a-2h}^{a}w^{q}(x)dx\Bigg)^{1/q}\Bigg(\sum\limits_{j=1}^{\infty}(2^{j}h)^{(\alpha-1)p'}
\int\limits_{a-2^{j+1}h}^{a-2^{j}h}w^{-p'}(x)dx\Bigg)^{1/p'}\\
&=&c\Bigg[\sum\limits_{j=1}^{\infty}(2^{j}h)^{(\alpha-1)p'}\Bigg(\int\limits_{a-2h}^{a}w^{q}(x)dx\Bigg)^{p'/q}\Bigg(
\int\limits_{a-2^{j+1}h}^{a-2^{j}h}w^{-p'}(x)dx\Bigg)\Bigg]^{1/p'}\\
&=&c\Bigg[\sum\limits_{j=1}^{\infty}\Bigg(\frac{\int\limits_{a-2h}^{a}w^{q}(x)dx}{\int\limits_{a-2^jh}^{a}w^{q}(x)dx}
\Bigg)^{p'/q}\Bigg(\frac{1}{2^jh}\int\limits_{a-2^jh}^{a}w^{q}(x)dx\Bigg)^{p'/q}
\Bigg(\frac{1}{2^jh}\int\limits_{a-2^{j+1}h}^{a-2^jh}w^{q}(x)dx
\Bigg)\Bigg]^{1/p'}\\
&\le& c\|w\|_{A_{p,q}^{-}(\Bbb{R})}^{1/q}\Bigg(\sum\limits_{j=1}^{\infty}\Bigg(\frac{\int\limits_{a-2h}^{a}w^{q}(x)dx}
{\int\limits_{a-2^jh}^{a}w^{q}(x)dx}\Bigg)^{p'/q}\Bigg)^{1/p'}\\
&\le& c\|w\|_{A_{p,q}^{-}(\Bbb{R})}^{1/q}\Bigg(\sum\limits_{j=0}^{\infty}\Bigg(\frac{\|w\|_{A_{p,q}^{-}(\Bbb{R})}}{1+\|w\|_{A_{p,q}^{-}(\Bbb{R})}}
\Bigg)^{p'j/q}\Bigg)^{1/p'}\\
&=&c\|w\|_{A_{p,q}^{-}(\Bbb{R})}^{1/q}\Bigg(\frac{1}{1-(\frac{\|w\|_{A_{p,q}^{-}(\Bbb{R})}}{1+\|w\|_{A_{p,q}^{-}(\Bbb{R})}}
)^{p'/q}}\Bigg)^{1/p'}\\
&\le&c\|w\|_{A_{p,q}^{-}(\Bbb{R})}^{1/q}\|w\|_{A_{p,q}^{-}(\Bbb{R})}^{1/p'}=c\|w\|_{A_{p,q}^{-}(\Bbb{R})}^{1-\alpha}.
\end{eqnarray*}

In the latter  inequality we used the fact that $\|w\|_{A^{-}_{p,q}({\Bbb{R}})}\geq 1$ and the relation
$ \Big(1- \big( \frac{s}{1+s}\big)^{p'/q}\Big)^{-1} = O(s)$ as $s\to \infty$.
\end{proof}




\subsection{Proofs of the main results}

{\em Proof of Theorem \ref{m10}.} (i) {\em Sharpness.} Let us take $0<\epsilon<1$. Suppose that $w(x)=|1-x|^{(1-\epsilon)/p'}$. Then it is easy to check that
$$\|w\|_{A_{p,q}^{+}(\Bbb{R})}=\|w^q\|_{A^{+}_{1+q/p'}(\Bbb{R})}\approx\frac{1}{\epsilon^{q/p'}}.$$
Observe also that for $f(t)=(1-t)^{\epsilon-1}\chi_{(0,1)}(t)$, $\|wf\|_{L^p}\approx c\frac{1}{\epsilon^{1/p}}$.
Now let $0<x<1$. Then we find that
$$M^{+}_{\alpha}f(x)\ge\frac{1}{(1-x)^{1-\alpha}}\int\limits_{x}^{1}f(t)dt=c\frac{1}{\epsilon}(1-x)^{\epsilon-1+\alpha}.$$
Finally,
\begin{equation}\label{m68}
\|wM^{+}_{\alpha}f\|_{L^{q}(\Bbb{R})}\ge \epsilon^{-1-1/q}.
\end{equation}
Thus letting $\epsilon\rightarrow0$ we conclude that the exponent $p'/q(1-\alpha)$ is sharp.

Inequality \eqref{SharpMax} follows from Proposition \ref{GaKo1}, (part (i))    Lemma \ref{m45} and the following  observations:
$$ A^+_{GK}(w^q, w^p, p,q)= \big[ w^{-p'}, w^{-q'}\big]_{Glo} \ll \| w^{-1}\|_{A^-_{q',p'}}^{1-\alpha}
= \|w\|_{A^{+}_{p,q}}^{(1-\alpha)p'/q}.$$
Notice also that $\frac{1}{q'}-\frac{1}{p'}=\alpha$.
\vskip+0.1cm

For part (ii), we show only the sharpness. Inequality \eqref{SharpMax1} follows in the same way as \eqref{SharpMax} was proved. Let $0<\epsilon<1$. Let $w(x)=|x|^{(1-\epsilon)/p'}$. Then
$$\|w\|_{A_{p,q}^{-}(\Bbb{R})}=\|w^q\|_{A^{-}_{1+q/p'}(\Bbb{R})}\approx\frac{1}{\epsilon^{q/p'}}.$$
Suppose that $f(t)=t^{\epsilon-1}\chi_{(0,1)}(t)$.
Suppose that $\|wf\|_{L^p}\approx c\frac{1}{\epsilon^{1/p}}$. If $0<x<1$, then we find that
$$M^{-}_{\alpha}f(x)\ge\frac{1}{x^{1-\alpha}}\int\limits_{0}^{x}f(t)dt=c\frac{1}{\epsilon}x^{\epsilon-1+\alpha}.$$
Finally,
\begin{equation}\label{m67}
\|wM^{-}_{\alpha}f\|_{L^{q}(\Bbb{R})}\ge \epsilon^{-1-1/q}.
\end{equation}
Letting $\epsilon\rightarrow0$ we conclude that the exponent $p'/q(1-\alpha)$ is sharp.
$\;\;\; \Box$
\vskip+0.2cm

{\em Proof} of Theorem \ref{m30} follows immediately from  Lemma \ref{m45}, Theorems \ref{m21} and \ref{m20}.
\vskip+0.2cm

{\em Proof} of Theorem \ref{m32}. We prove only \eqref{m31b}; the proof of the estimate \eqref{m31} follows analogously. \\
For estimate \eqref{m31b}, observe that Theorem \ref{m22} and Theorem \ref{m30} yield
\begin{eqnarray*}
\|\mathcal{R}_{\alpha}\|_{L^p_{w^p}\rightarrow L^q_{v^q}}&\ll& \Big( \|w\|_{A_{p,q}^{-}(\Bbb{R})}^{1-\alpha}+ A^-_{GK}(w^q, w^p, p,q) \Big) \ll \Big( \|w\|_{A_{p,q}^{-}(\Bbb{R})}^{1-\alpha}+ \|w\|_{A_{p,q}^{-}(\Bbb{R})}^{p'/q(1-\alpha)}\Big) \\
&\ll& \|w\|_{A_{p,q}^{-}(\Bbb{R})}^{(1-\alpha)\max\{1,p'/q\}}.
\end{eqnarray*}
Here we used the relations:
\begin{eqnarray*}
A^-_{GK}(w^q, w^p, p,q) \ll  \|w^{-1}\|^{1-\alpha}_{A_{q',p'}^{+}(\Bbb{R})}=c\|w\|^{p'/q(1-\alpha)}_{A_{p,q}^{-}(\Bbb{R})}.
\end{eqnarray*}
The first inequality is  due to Lemma \ref{m45}, while the latter equality is easy to check.

For sharpness first we assume that $p'/q\ge1$. Observe that the following pointwise estimate
$$M_{\alpha}^{-}f\leq \mathcal{R}_{\alpha}f,\indent\indent f\ge0$$
holds.

Then by using same $w$ and $f$ as in the proof of Theorem \ref{m10} we have estimate \eqref{m67} with $M_{\alpha}^{-}$ replaced by $\mathcal{R}_{\alpha}$ showing sharpness. Similarly for $\mathcal{W}_{\alpha}$ observe that the following pointwise estimate:
$$M_{\alpha}^{+}f\leq \mathcal{W}_{\alpha}f,\indent\indent f\ge0$$
holds.

Then by applying the same $w$ and $f$ as in the proof of (ii) of Theorem \ref{m10}, we have estimate \eqref{m68} with $M_{\alpha}^{+}$ replaced by $\mathcal{W}_{\alpha}$ shows sharpness. The case when $p'/q<1$ follows from duality argument. \\
Indeed,
$$\|\mathcal{R}_{\alpha}\|_{L^{p}_{w^p}\rightarrow L^{q}_{w^q}}=\|\mathcal{W}_{\alpha}\|_{L^{q'}_{w^{-q'}}\rightarrow L^{p'}_{w^{-p'}}}\le c\|w^{-1}\|^{(1-\alpha)q/p'}_{A^{+}_{q',p'}(\Bbb{R})}=c\|w\|_{A_{p,q}^{-}(\Bbb{R})}^{1-\alpha}.$$

$\Box$
\section{Strong Maximal  and Multiple Integral Operators}\label{Strong}

Let $T$ be an operator acting on a class of functions defined on $\Bbb{R}^n$. We denote by $T^k$, $k=1\cdots n$, the operator defined on functions in $\Bbb{R}$ by letting $T$ act on the $k$-th variable while keeping the remaining variable fixed. Formally,
$$(T^k f)(x)=(Tf(x_1,x_2\cdots,x_{k-1},\cdot,x_{k+1},\cdots,x_n)(x_k).$$
\begin{defin}\label{m12}
A weight function $w$ satisfies $A_p^{(s)}(\Bbb{R}^n)$ condition $(w\in A_{p}^{(s)}(\Bbb{R}^n))$, $1<p<\infty$, if
$$\|w\|_{A_{p}^{(s)}(\Bbb{R}^n)}:=\sup\limits_{P}\bigg(\frac{1}{|P|}\int\limits_{P}w(x)dx\bigg)\bigg(\frac{1}{|P|}\int\limits_{P}w(x)^{-1/p-1}dx\bigg)^{p-1}<\infty,$$
where the supremum is taken over all parallelepipeds $P$ in $\Bbb{R}^{n}$ with sides parallel to the co-ordinate axes.
\end{defin}
\begin{defin}
Let  $1<p<\infty$. A weight function $w=w(x_1,\cdots,x_n)$ defined on ${\Bbb{R}}^n$ is said to satisfy $A_p$ condition in $x_{i}$ uniformly with respect to other variables $(w\in A_{p}(x_i))$ if
\begin{eqnarray*}
\|w\|_{A_{p}(x_i)}:= \esssup\limits_{\substack{(x_{1},\cdots x_{i-1},\\x_{i+1}\cdots,x_{n})\in{\Bbb{R}}^{n-1}}}\sup_{I} &&\bigg(\frac{1}{|I|}\int\limits_{Q}w(x_{1},\cdots,x_{n})dx_{i}\bigg)\times\\
&&\times\bigg(\frac{1}{|I|}\int\limits_{I}w^{1-p'}(x_{1},\cdots,x_{n}) dx_{i}\bigg)^{p-1}<\infty,
\end{eqnarray*}
where by $I$ we denote a bounded interval in $\Bbb{R}$.
\end{defin}

\begin{rem} $w(x_1,\cdots,x_n)\in A_{p}^{(s)}(\Bbb{R}^n)\Leftrightarrow w\in\bigcap\limits_{i=1}^{n} A_{p}(x_i)$ (see e.g., pp. 453-454 of \cite{GARC}, \cite{KOK}).
\end{rem}
\begin{prop}\label{m12}
Let $T$ be an operator defined on $\Bbb{R}^{n}$ such that
\begin{equation}\label{m13}
\|T^k\|_{L^{p}_{w}(\Bbb{R})}\leq c\|w\|_{A_p(x_k)}^{\gamma(p)}\indent\indent k=1,\cdots,n,
\end{equation}
holds, where $\gamma(p)$ is a constant depending only on $p$. Then the following estimate
$$\|T\|_{L^{p}_{w}(\Bbb{R}^{n})}\leq c(\|w\|_{A_p(x_{1})}\cdots \|w\|_{A_p(x_{n})} )^{\gamma(p)}.$$
holds.
\end{prop}
\begin{proof}
For simplicity we give proof for $n=2$ the proof general case is the same. Suppose that $f\ge0$. Using \eqref{m12} two times and Fubini's theorem we have,
\begin{eqnarray*}
\|Tf\|^p_{L^p_w(\Bbb{R}^2)}&=&\int\limits_{\Bbb{R}}\int\limits_{\Bbb{R}}(Tf(x_{1},x_{2}))^{p}w(x_{1},x_{2})dx_{1}dx_{2}\\
&=&\int\limits_{\Bbb{R}}\Bigg(\int\limits_{\Bbb{R}}(T^{1}(T^{2}f(\cdot,x_{2})))(x_{1})^{p}w(x_{1},x_{2})dx_{1}\Bigg)dx_{2}\\
&\le&c\|w\|_{A_{p}(x_{1})}^{p\gamma(p)}\int\limits_{\Bbb{R}}\Bigg(\int\limits_{\Bbb{R}}(T^{2}f(x_{1},x_{2}))^{p}w(x_{1},x_{2})dx_{1}\Bigg)dx_{2}\\
&=&c\|w\|_{A_{p}(x_{1})}^{p\gamma(p)}\int\limits_{\Bbb{R}}\Bigg(\int\limits_{\Bbb{R}}(T^{2}f(x_{1},x_{2}))^{p}w(x_{1},x_{2})dx_{2}\Bigg)dx_{1}\\
&=&c(\|w\|_{A_{p}(x_{1})}\|w\|_{A_{p}(x_{2})})^{p\gamma(p)}\|f\|^p_{L^{p}_{w}(\Bbb{R}^{2})}
\end{eqnarray*}
\end{proof}

\subsection{Strong Hardy--Littlewood Maximal Functions and Multiple Singular Integrals}

The following theorem is due to S. M. Buckley \cite{Buck}.
\begin{thma}\label{m66}
If $w\in A_p({\Bbb{R}}^n)$, then $\|Mf\|_{L^{p}_{w}({\Bbb{R}}^n)}\le c_{n,p}\|w\|_{A_{p}({\Bbb{R}}^n)}^{1/(p-1)}\|f\|_{L^{p}_{w}({\Bbb{R}}^n)}$. The power $1/(p-1)$ is best possible.
\end{thma}
Let $f$ be a locally integrable function on $\Bbb{R}^{n}$. Then we define strong Hardy--Littlewood maximal operator as
$$\big(M^{(s)}f\big)(x)=\sup\limits_{P\ni x}\frac{1}{|P|}\int\limits_{P}|f(y)|dy,\indent\indent x\in\Bbb{R}^{n},$$
where the supremum is taken over all parallelepipeds $P\ni x$ in $\Bbb{R}^{n}$ with sides parallel to the co-ordinate axes.

\begin{thm}\label{m14}
Let $1<p<\infty$ and $w$ be a weight function on $\Bbb{R}^{n}$ such that $w\in  A^{(s)}_{p}({\Bbb{R}}^n)$. Then there exists a constant $c$ depending only on $n$ and $p$ such that the following inequality
\begin{equation}\label{m3}
\|M^{(s)}f\|_{L^{p}_{w}({\Bbb{R}}^n)}\le c\Bigg(\prod\limits_{i=1}^{n}\|w\|_{A_p(x_{i})}\Bigg)^{1/(p-1)}\|f\|_{L^{p}_{w}({\Bbb{R}}^n)}
\end{equation}
holds, for all $f\in L^{p}_w({\Bbb{R}}^n)$ . Further, the power $1/(p-1)$ in estimate \eqref{m3} is sharp.
\end{thm}
\begin{proof}
We can estimate $M^{(s)}$ as follows
$$\big( M^{(s)}f\big) (x)\leq \big( M^{1}\circ M^{2}\circ\dots\circ M^{n}\big) f(x),$$
where
$$\big( M^{k}f\big) (x_{1},\cdots x_{n})=\sup\limits_{ I_{k}\ni x_{k}}\frac{1}{|I_k|}\int\limits_{I_k}|f(x_{1},\cdots,x_{k-1},t,x_{k+1},\cdots,x_{n})|dt.$$
Now by Theorem \ref{m66} (for $n=2$) we have $\gamma(p)=\frac{1}{p-1}$ and consequently, by Proposition \ref{m12} we find that
$$\|M^{(s)}f\|_{L^{p}_{w}({\Bbb{R}}^n)}\le c\Bigg(\prod\limits_{i=1}^{n}\|w\|_{A_p(x_{i})}\Bigg)^{1/(p-1)}\|f\|_{L^{p}_{w}({\Bbb{R}}^n)}.$$
For sharpness we consider the case for $n=2$. Observe that when $w$ is of product type,  i.e. $w(x_{1},x_{2})=w_{1}(x_{1})w_{2}(x_{2})$, then
\begin{equation} \label{PR}
\|w\|_{A_{p}(x_{1})}=\|w_1\|_{A_{p}(\Bbb{R})},\indent\indent \|w\|_{A_{p}(x_{2})}=\|w_2\|_{A_{p}(\Bbb{R})}.
\end{equation}
Let us take $0<\epsilon<1$. Suppose that $w(x_{1},x_{2})=|x_{1}|^{(1-\epsilon)(p-1)}|x_{2}|^{(1-\epsilon)(p-1)}$. Then it is easy to check that $$(\|w\|_{A_{p(x_1)}}\|w\|_{A_{p(x_2)}})^{1/(p-1)}\approx\frac{1}{\epsilon^2}.$$
Observe also that for
$$f(x_{1},x_{2})={x_{1}}^{\epsilon(p-1)-1}\chi_{(0,1)}(x_{1}){x_{2}}^{\epsilon(p-1)-1}\chi_{(0,1)}(x_{2}),$$
we have $\|f\|^p_{L^p_w}\approx\frac{1}{\epsilon^2}$.
Now let $0<x_{1},x_{2}<1$. Then we find that the following estimate
$$\big( M^{(s)}f\big) (x_{1},x_{2})\ge\frac{1}{x_{1}x_{2}}\int\limits_{0}^{x_{1}}\int\limits_{0}^{x_{2}}f(t,\tau)dtd\tau=c\frac{1}{\epsilon^{2}}f(x,y)$$
holds. Finally
$$\|M^{(s)}f\|_{L^{p}_w}\ge c\frac{1}{\epsilon^2}\|f\|_{L^{p}_{w}}$$
Thus we have the sharpness in  \eqref{m3}.
\end{proof}
Let us denote by ${\mathcal{H}}^{(n)}$ the Hilbert transform with product kernels (or $n$-dimensional Hilbert transform) defined by $$\big( {\mathcal{H}}^{(n)}f\big) (x)=\lim\limits_{\substack{\epsilon_{1}\rightarrow0\\\cdots\\ \epsilon_{n}\rightarrow0}}\indent\int\limits_{|x_{1}-t_{1}|>\epsilon_{1}}\cdots\int\limits_{|x_{n}-t_{n}|>\epsilon_{n}}\frac{f(t_{1},\cdots,t_{n})}{(x_{1}-t_{1})
\cdots(x_{n}-t_{n})}dt_{1}\cdots dt_{n}.$$
We denote ${\mathcal{H}}^{(1)}=: { {\mathcal{H}} }$. Notice that for each $x\in\Bbb{R}^{n}$, we can write
\begin{equation}\label{m69}
\big( {\mathcal{H}}^{(n)}f\big) (x)=\big(\mathcal{H}^{1}\circ\cdots\circ \mathcal{H}^{n}\big)f(x)
\end{equation}
where,
$$\big( \mathcal{H}^{k}f\big) (x)=\lim\limits_{\epsilon_{k}\rightarrow 0} \int\limits_{|x_{k}-y_{k}|>\varepsilon_k} \frac{f(x_1, \cdots, y_k, \cdots, x_n)}{x_{k}-y_{k}}dy.$$
The following theorem is due to S. Petermichl \cite{PET1}.
\begin{thma}\label{m70}
Let $1<p<\infty$.  Then there exist positive constant $c$ depending only on  $p$ such that for all weights $w\in A_p({\Bbb{R}})$ we have
\begin{equation}\label{f2}
\|\mathcal{H}f\|_{L^{p}_w({\Bbb{R}})}\leq c\|w\|_{A_{p}({\Bbb{R}})}^{\beta}\|f\|_{L^p_w({\Bbb{R}})}, \;\;\; f\in L^p_w({\Bbb{R}}),
\end{equation}
where $\beta=\max\{1,p'/p\}$. Moreover, the exponent $\beta$ in this estimate is sharp.
\end{thma}
 \begin{thm}\label{m15}
Let $1<p<\infty$ and $w$ be a weight function on $\Bbb{R}^{n}$ such that $w\in{A^{(s)}_{p}}(\Bbb{R}^n)$. Then there exists a constant $c$ depending only on $n$ and $p$ such that the following inequality
\begin{equation}\label{m16}
\|{\mathcal{H}}^{(n)}f\|_{L^{p}_{w}({\Bbb{R}}^n)}\le c(\|w\|_{A_p(x_{1})}\cdots\|w\|_{A_p(x_{n})})^{\max\{1,p'/p\}}\|f\|_{L^{p}_{w}({\Bbb{R}}^n)}
\end{equation}
holds for all $f\in L^{p}_w$ . Further the power ${\max\{1, p'/p\}}$  in estimate \eqref{m16} is sharp.
\end{thm}
\begin{proof}
Using representation \eqref{m69},  Proposition \ref{m12} and Theorem \ref{m70}, we have that
$$\|{\mathcal{H}}^{(n)}f\|_{L^{p}_{w}({\Bbb{R}}^n)} \ll\big(\|w\|_{A_p(x_{1})}\cdots\|w\|_{A_p(x_{n})}\big)^{\max\{1, p'/p\}}\|f\|_{L^{p}_{w}({\Bbb{R}}^n)}.$$
Let $n=2$. For sharpness we observe that when $w$ is of product type i.e. $w(x_{1},x_{2})=w_{1}(x_{1})w_{2}(x_{2})$, then inequality \eqref{PR} holds.
Let us first derive sharpness for $p=2$. Let us take $0<\epsilon<1$ and let $w(x_{1},x_{2})=w_1(x_1)w_2(x_2)$, where $w_1(x_1)= |x_{1}|^{1-\epsilon}$ and $w_2(x_2)=|x_{2}|^{1-\epsilon}$. Then it is easy to check that \eqref{PR} holds. Observe also that for

\begin{equation}\label{m78}
f(x_{1},x_{2})={x_{1}}^{\epsilon-1}\chi_{(0,1)}(x_{1}){x_{2}}^{\epsilon-1}\chi_{(0,1)}(x_{2}),
\end{equation}
$\|f\|^2_{L^2_w}\approx\frac{1}{\epsilon}.$
Now let $0<x_{1},x_{2}<1$. Then we find that
$$\|{\mathcal{H}}^{(2)}f\|_{L^{2}_w({\Bbb{R}}^2)}\ge 4\epsilon^{-3}.$$
Letting $\epsilon\rightarrow0$ we have sharpness in  \eqref{m16} for $p=2$ i.e., the estimate
$$\|{\mathcal{H}}^{(2)}\|_{L^{2}_{w}({\Bbb{R}}^2)}\ll\|w\|_{A_{2}(x_1)}\|w\|_{A_{2}(x_2)}$$
is sharp.\\
Let $1<p<2$. Suppose that $0<\epsilon<1$ and that $w(x_{1},x_{2})=|x_{1}|^{(1-\epsilon)}|x_{2}|^{(1-\epsilon)}$. Then it is easy to check that $$(\|w\|_{A_{p(x_1)}}\|w\|_{A_{p(x_2)}})^{1/(p-1)}\approx\frac{1}{\epsilon^2}.$$
Observe also that for the function defined by \eqref{m78} the relation
$\|f\|_{L^p_w}\approx(\frac{1}{\epsilon^2})^{\frac{1}{p}}$ holds.
Now let $0<x_{1},x_{2}<1$. Then we find that following estimates
$$\|{\mathcal{H}}^{(2)}f\|_{L^{p}_w({\Bbb{R}}^2)}\ge \frac{1}{\epsilon^2} \|f\|_{L^p_w({\Bbb{R}}^2)} \approx\big( \|w\|_{A_{p}(x_1)} \|w\|_{A_{p}(x_2)}\big)^{p'/p}\|f\|_{L^p_{w}({\Bbb{R}}^2)}$$
are fulfilled. Thus we have sharpness in  \eqref{m16} for $1<p<2$. Using the fact that $n$-dimensional Hilbert transform is essentially self-adjoint and applying duality argument together with the obvious equality
$$\|u^{1-p'}\|_{A_{p'}}=\|u\|^{1/(p-1)}_{A_p}, \;\;\; u\in A_p, $$
we have sharpness for $p>2$. This completes the proof.
\end{proof}

Let $x=(x^{(1)},\cdots,x^{(n)})\in\Bbb{R}^{d_{1}}\times\cdots\times\Bbb{R}^{d_{n}}$, where $d_{1},d_{2},\cdots, d_{n}\in\Bbb{N}$. Suppose that $x^{(k)}_{j_k}$ are components of $x^{(k)}$, $k=1, \cdots, n$, $1\leq j_k \leq d_k$. Then we define $n$-fold Riesz transform
\begin{eqnarray*}
\big( R_{(j_{1},\cdots,j_{n})}^{(n)}f\big) (x)=p.v.\int\limits_{\Bbb{R}^{d_{1}}}\!\!\!\cdots\!\!\!\int\limits_{\Bbb{R}^{d_{n}}}
\prod_{k=1}^{n}\frac{(x_{j_{k}}^{(k)}-y_{j_{k}}^{(k)})}{|x^{(k)}-y^{(k)}|^{d_{k}+1}}f(y^{(1)},\cdots,y^{(n)})dy^{(1)}\cdots dy^{(n)},
\end{eqnarray*}
where $1\le j_{k}\le d_k$, $k=1,\cdots ,n$. It can be noticed that
\begin{equation*}
\big( R_{(j_{1},\cdots,j_{n})}^{(n)}f\big) (x)=\big( R^{1}_{j_{1}}\circ\dots\circ R^{n}_{j_{n}}f\big) (x)
\end{equation*}
where
\begin{eqnarray*}\big( R_{(j_{1},\cdots,j_{n})}^{k}f\big) (x)=p.v\int\limits_{\Bbb{R}^{d_{k}}}
\!\!\!\frac{x_{j_{k}}^{(k)}-y_{j_{k}}^{(k)}}{|x^{(k)}-y^{(k)}|^{d_{k}+1}}f(x^{(1)},\cdots, x^{(k-1)}, y^{(k)}, x^{(k+1)}\cdots, x^{(n)})dy^{(k)}.
\end{eqnarray*}

\begin{thm}\label{m17}
Let $1<p<\infty$ and $w$ be a weight function on $\Bbb{R}^{d_{1}}\times\cdots\times\Bbb{R}^{d_{n}}$ satisfy the condition $w\in{A^{(s)}_{p}({\Bbb{R}}^{d_1}\times \cdots \times {\Bbb{R}}^{d_n})}$. Then there exist a constant $c$ depending only on $n$ and $p$ such that the following inequalities
\begin{equation}\label{m18}
\|R^{(n)}_{(j_{1},\cdots,j_{n})}f\|_{L^{p}_{w}({\Bbb{R}}^{d})}\le c(\|w\|_{A_p(x^{(1)})}\cdots\|w\|_{A_p(x^{(n)})})^{\max\{1,p'/p\}}\|f\|_{L^{p}_{w}({\Bbb{R}}^{d})}
\end{equation}
hold for all $f\in L^{p}_w({\Bbb{R}}^{d})$ and $1\leq j_{k}\leq d_{k}$, $k=1\cdots n$, where $d= d_1+ \cdots, +d_n$. Further, the power ${\max\{1, p'/p\}}$  in estimate \eqref{m18} is sharp.
\end{thm}
{\em Proof} of this statement is similar to that of the previous one; we need to apply Proposition \ref{m12} and the results of \cite{PET2}. $\;\;\; \Box$.

\begin{exa}
Let $-1<\gamma<p-1$. It is known that $w(x)=|x|^{\gamma}$ belongs to $A_p^{(s)}(\Bbb{R}^n)$.  Let  $\overline{w}(t)=|t|^\gamma$,\;\; $t\in\Bbb{R}$. We set $b_\gamma:=\max\{2^{\frac{\gamma}{2}},1\}$ and $d_\gamma:=\max\{2^{\frac{-\gamma}{2}-p+1},1\}$.\\
(i) Then  it follows from  Theorem \ref{m14} that
\begin{equation}\label{z09}
\|M^{(s)}\|_{L^p_w(\Bbb{R}^n)}\le C^n {C_{\gamma}}^{\frac{n}{p-1}} \bigg(1+\|\overline{w}\|_{A_p(\Bbb{R})}\bigg)^{\frac{n}{p-1}}
\end{equation}
where $C$ is the constant from the Buckley's estimate (see \eqref{f1}) and
\begin{equation}\label{z091}
C_\gamma = \left\{
\begin{array}{l}
b_\gamma, \;\; 0\leq \gamma < p-1, \\
d_\gamma, \;\; -1<\gamma<0;
\end{array}
\right.
\end{equation}

It is known (see \cite{GRA} pp 287--289)  that $C$ in \eqref{z09} can be taken as $C=3^{p+p'}2^{p'-p}p'24^{\frac{2}{p}}p^{\frac{1}{p-1}}$. \\
(ii) From Theorem \ref{m16} it follows that
$$\|\mathcal{H}^{(n)}\|_{L^p_w(\Bbb{R}^n)}\le c^n C_{\gamma}^{{n}\max\{ 1, p'/p \}} \bigg(1+\|\overline{w}\|_{A_p(\Bbb{R})}\bigg)^{n \max\{1,\frac{p'}{p}\}}$$
holds, where $c$ is the constant from \eqref{f2} and $C_\gamma$ is defined in \eqref{z091}.\\
Following \cite{GRA}, pp 285--286, it can be verified that $$\|\overline{w}\|_{A_p(\Bbb{R})}\le \max\bigg\{2^{|\gamma|},\frac{4^p}{(\gamma+1)(\gamma(1-p')+1)^{p-1}}\bigg\}.$$
We can get also another type of estimate of the norms in $L^p_w$. By using the same arguments as in \cite{GRA}, pp 285--286, we find that
$$\|w\|_{A_p(x_i)}\le\left\{
\begin{array}{l}
\Gamma_\gamma, \;\; 0\leq \gamma < p-1, \\
G_\gamma, \;\; -1<\gamma<0;
\end{array}
\right.\indent i=1,2,$$
where $$\Gamma_\gamma=\max\bigg\{((4/3)^2+1)^{\gamma/2}(2/3)^{\gamma},  b_\gamma4^p\bigg((\gamma+1)^{-1}(\gamma(1-p')+1)^{1-p}+1\bigg)\bigg\},$$
 $$G_\gamma=\max\bigg\{((4/3)^2+1)^{-\gamma/2}(2/3)^{\gamma},
d_0 d_\gamma4^p\bigg((\gamma+1)^{-1}(\gamma(1-p')+1)^{1-p}+1\bigg)\bigg\}.$$
Consequently, using directly Theorems \ref{m14} and \ref{m16} we have the following estimate
$$\|M^{(s)}\|_{L^{p}_w(\Bbb{R}^n)}\leq
C^n\left\{
\begin{array}{l}
\Gamma_\gamma^{n/(p-1)}, \;\; 0\leq \gamma < p-1, \\
G_\gamma^{n/(p-1)}, \;\; -1<\gamma<0;
\end{array}
\right.$$

$$\|\mathcal{H}^{(n)}\|_{L^{p}_w(\Bbb{R}^n)}\leq
c^n\left\{
\begin{array}{l}
\Gamma_\gamma^{n\max\{ 1, p'/p \}}, \;\; 0\leq \gamma < p-1, \\
G_\gamma^{n\max\{ 1, p'/p \}}, \;\; -1<\gamma<0;
\end{array}
\right.$$
where $C$ and $c$ are constants in \eqref{f1} and \eqref{f2} respectively.

\end{exa}

\subsection{Strong Fractional Maximal Functions and Reisz Potentials with Product Kernels}

In this subsection we state and prove sharp weighted norm estimates for strong fractional maximal and Reisz potential with product kernels. To get the main results we use the ideas of the previous subsection.

\begin{defin}\label{m48}
A weight function $w$ satisfies $A_{p,q}^{(s)}$ condition $(w\in A_{p,q}^{(s)})$, $1<p<\infty$ if
$$\|w\|_{A_{p,q}^{(s)}}:=\sup\limits_{P\ni x}\bigg(\frac{1}{|P|}\int\limits_{P}w^{q}(x)dx\bigg)^{1/q}\bigg(\frac{1}{|P|}\int\limits_{P}w(x)^{-p'}dx\bigg)^{1/p'}<\infty,$$
where the supremum is taken over all parallelepipeds $P$ in $\Bbb{R}^{n}$ with sides parallel to the co-ordinate axes.
\end{defin}
\begin{defin}
Let $1<p\leq q<\infty$. A weight function $w=w(x_1,\cdots,x_n)$ defined on ${\Bbb{R}}^n$ is said to satisfy $A_{p,q}$ condition in $x_{i}$ uniformly with respect to other variables $(w\in A_{p,q}(x_i))$ if
\begin{eqnarray*}
\|w\|_{A_{p,q}(x_i)}:= \esssup\limits_{\substack{(x_{1},\cdots, x_{i-1},\\x_{i+1}\cdots,x_{n}) \in{\Bbb{R}}^{n-1}}}\sup_{I\ni x_{i}} &&\bigg(\frac{1}{|I|}\int\limits_{I}w^{q}(x_{1},\cdots,x_{n})dx_{i}\bigg)^{1/q}\times\\
&&\times\bigg(\frac{1}{|I|}\int\limits_{I}w^{-p'}(x_{1},\cdots,x_{n})dx_{i}\bigg)^{1/p'}<\infty,
\end{eqnarray*}
where $I$ is a bounded interval.
\end{defin}
\begin{rem} Like $A^{(s)}_p(\Bbb{R}^n)$ weights for given $w(x_1,\cdots,x_n)\in A_{p,q}^{(s)}\Leftrightarrow w\in\bigcap\limits_{i=1}^{n} A_{p,q}(x_i)$.
\end{rem}

\begin{prop}\label{m49}
Let $1<p\le q<\infty$ and let $T$ be an operator defined on $\Bbb{R}^{n}$. Suppose that $w$ is weight, $w\in A_{p,q}^{(s)}$. Let
\begin{equation}\label{m50}
\|T^k\|_{L^{p}_{w^p}(\Bbb{R})\rightarrow L^{q}_{w^q}(\Bbb{R})}\leq c\|w\|_{A_{p,q}(x_k)}^{\gamma(p,q)}\indent\indent k=1,\cdots,n,
\end{equation}
hold, where $\gamma(p,q)$ is a constant depending only on $p$ and $q$. Then
$$\|T\|_{L^{p}_{w^p}(\Bbb{R}^{n})\rightarrow L^{q}_{w^q}(\Bbb{R}^{n})}\leq c(\|w\|_{A_p(x_{1})}\cdots \|w\|_{A_p(x_{n})} )^{\gamma(p,q)}.$$
\end{prop}
Proof  is similar to that of Proposition \ref{m12}; therefore it is omitted.
\vskip+0.2cm
The following Theorem is from   \cite{LAC-MOE-PER-TORR}.
\begin{thma}\label{m72}
Suppose that $0<\alpha<n$, $1<p<n/\alpha$ and $q$ is defined by the relationship $1/q=1/p-\alpha/n$. If $w\in A_{p,q}({\Bbb{R}}^n)$, then
$$\|wM_{\alpha}f\|_{L^q({\Bbb{R}}^n)}\le c\|w\|_{A_{p,q}({\Bbb{R}}^n)}^{\frac{p'}{q}(1-\alpha/n)}\|wf\|_{L^p({\Bbb{R}}^n)}.$$
Furthermore, the exponent ${\frac{p'}{q}(1-\alpha/n)}$ is sharp.
\end{thma}
Let $f$ be a locally integrable function and let $0<\alpha<1$. The strong fractional maximal operator is defined by
$$\big( M_{\alpha}^{(s)}f\big) (x)=\sup\limits_{ P\ni x }\frac{1}{|P|^{1-\alpha}}\int\limits_{P}|f(y)|dy,$$
where the supremum is taken over all parallelepipeds $P$ in $\Bbb{R}^{n}$ with sides parallel to the co-ordinate axes. It is easy to see that
\begin{equation}\label{m71}
\big( M_{\alpha}^{(s)}f\big) (x)\leq \big( M_{\alpha}^{1}\circ M_{\alpha}^{2}\circ\dots\circ M_{\alpha}^{n}f\big) (x),
\end{equation}
where
$$\big( M_{\alpha}^{k}f\big) (x_{1},\cdots x_{n})=\sup\limits_{I_{k}\ni x_{k}}\frac{1}{|I_k|^{1-\alpha}}\int\limits_{I_k}|f(x_{1},\cdots,x_{k-1},t,x_{k+1},\cdots,x_{n})|dt,$$
where $I_k$ are intervals in $\Bbb{R}$ such that $P=I_{1}\times\cdots\times I_{k}$.
\begin{thm}\label{m51a}
Let $0<\alpha<1$, $1<p<\frac{1}{\alpha}$, $q=\frac{p}{1-\alpha p}$ and $w$ be a weight function on $\Bbb{R}^{n}$ such that $w\in{A^{(s)}_{p,q}({\Bbb{R}}^n)}$. Then there exists a constant $c$ depending only on $n$, $p$ and $\alpha$ such that the following inequality
\begin{equation}\label{m52a}
\|wM_{\alpha}^{(s)}f\|_{L^{q}({\Bbb{R}}^n)}\le c\Bigg(\prod\limits_{i=1}^{n}\|w\|_{A_{p,q}(x_{i})}\Bigg)^{\frac{p'}{q}(1-\alpha)}\|wf\|_{L^{p}({\Bbb{R}}^n)}
\end{equation}
holds, for all $f\in L^{p}_{w^p}({\Bbb{R}}^n)$ . Further, the power ${\frac{p'}{q}(1-\alpha)}$ in estimate \eqref{m52a} is sharp.
\end{thm}
\begin{proof}
Using estimate \eqref{m71}, Theorem \ref{m72} and Proposition \ref{m49} we get easily \eqref{m52a}.
The main ''difficulty'' here is to derive sharpness. Let, for simplicity, $n=2$. Let us take $0<\epsilon<1$. suppose that $w$ is of product type $w(x_{1},x_{2})=w_1(x_1) w_2(x_2)$, where $w_1(x_1)= |x_{1}|^{(1-\epsilon)/p'}$ and $w_2(x_2)= |x_{2}|^{(1-\epsilon)/p'}$. Then it is easy to see that
$$\|w\|_{A_{p,q}(x_1)}=\|w_1\|_{A_{1+q/p'}({\Bbb{R}})}\approx\epsilon^{-q/p'}; \;\; \|w\|_{A_{p,q}(x_2)}=\|w_2\|_{A_{1+q/p'}({\Bbb{R}})}\approx\epsilon^{-q/p'}. $$
Further, if
$$f(t_{1},t_{2})=|t_{1}|^{\epsilon-1}\chi_{(0,1)}(t_{1})|t_{2}|^{\epsilon-1}\chi_{(0,1)}(t_{2}),$$
then $\|wf\|_{L^p({\Bbb{R}}^2)} \approx \frac{1}{\epsilon^{2/p}}$. Let $0<x_1,x_2<1$. Then we find that
$$M^{(s)}_{\alpha}f(x_{1},x_{2})\ge\frac{1}{|x_{1}|^{1-\alpha}|x_{2}|^{1-\alpha}}\int\limits_{0}^{x_{1}}\int\limits_{0}^{x_{2}}f(t_{1},t_{2})dt_{1}dt_{2}\approx c\frac{|x_{1}|^{\epsilon-1+\alpha}|x_{2}|^{\epsilon-1+\alpha}}{\epsilon^{2}}.$$
Finally we conclude that,
\begin{equation}\label{m55a}
\|wM^{(s)}_{\alpha}f\|_{L^{q}({\Bbb{R}}^2)}\ge \epsilon^{-2-2/q}.
\end{equation}
Thus letting $\epsilon\rightarrow0$ we have sharpness.
\end{proof}

Let $0<\alpha<1$. We define Riesz potential with product kernels on $\Bbb{R}^{n}$ as follows:
$$(I_{\alpha}^{(n)}f)(x)=\int\limits_{\Bbb{R}^{n}}\frac{f(t_1,\cdots,t_n)}{\prod\limits_{i=1}^{n}|x_{i}-t_{i}|^{1-\alpha}}dt_{1}\cdots dt_n, \;\;\; x=(x_{1},\cdots,x_{n}) \in {\Bbb{R}}^n.$$
When $n=1$ we use the symbol $I_{\alpha}$ for $I_{\alpha}^{(1)}$. The following Theorem is from  \cite{LAC-MOE-PER-TORR}.
\begin{thma}\label{m73}
Let $0<\alpha<n$, $1<p<n/\alpha$. We put $q=\frac{np}{n-\alpha p}$. Suppose that $w\in A_{p,q}({\Bbb{R}}^n)$. Then
$$\|wI_{\alpha}f\|_{L^{q}({\Bbb{R}}^n)}\leq c\|w\|_{A_{p,q}({\Bbb{R}}^n)}^{(1-\alpha/n)\max\{1,p'/q\}}\|wf\|_{L^p({\Bbb{R}}^n)}.$$
Furthermore, this estimate is sharp.
\end{thma}
Our result regarding $I_{\alpha}^{(n)}$ reads as follows:
\begin{thm}\label{m53}
Let $0<\alpha<1$, $1<p<1/\alpha$. We put $q= \frac{p}{1-\alpha p}$. Let $w$ be a weight function on $\Bbb{R}^{n}$ such that $w\in{A^{(s)}_{p,q}({\Bbb{R}}^n)}$. Then there exists a constant $c$ depending only on $n$, $p$ and $\alpha$ such that the following inequality
\begin{equation}\label{m54}
\|wI^{(n)}_{\alpha}f\|_{L^{q}({\Bbb{R}}^n)}\le c\Bigg(\prod\limits_{i=1}^{n}\|w\|_{A_{p,q}(x_{i})}\Bigg)^{\max\{1,\frac{p'}{q}\}(1-\alpha)}\|wf\|_{L^{p}({\Bbb{R}}^n)}
\end{equation}
holds for all $f\in L^{p}_{w^p}({\Bbb{R}}^n)$ . Further, the power ${\max\{1,\frac{p'}{q}\}(1-\alpha)}$ in estimate \eqref{m54} is sharp.
\end{thm}

Proof of this statement follows using the same arguments as in the proof of Theorem \ref{m51a} together with Theorem \ref{m73}.

\subsection{Strong one--sided maximal operators}
Based on the results derived in this paper we give sharp bounds for one--sided strong maximal operators.

Let $f$ be locally integrable function on $\Bbb{R}^{n}$. We define one--sided strong fractional maximal operators as
\begin{equation}\label{m56}
M^{+(s)}f(x_{1},\cdots,x_{n})=\frac{1}{\prod\limits_{i=1}^{n}{h_i}}
\int\limits_{x_{1}}^{x_{1}+h_{1}}\cdots\int\limits_{x_{n}}^{x_{n}+h_{n}}|f(y_1,\cdots y_n)|dy_1\cdots dy_{n},
\end{equation}
\begin{equation}\label{m57}
M^{-(s)}_{\alpha}f(x_{1},\cdots,x_{n})=\frac{1}{\prod\limits_{i=1}^{n}{h_i}}
\int\limits_{x_{1}-h_{1}}^{x_{1}}\cdots\int\limits_{x_{n}-h_n}^{x_{n}}|f(y_1,\cdots y_n)|dy_1\cdots dy_{n}.
\end{equation}

Let $1<p<\infty$. We say that a weight function $w$ belongs to the class $A_{p}^{-(s)}(\Bbb{R}^{n})$ if
\begin{eqnarray*}
\|w\|_{A_{p}^{-(s)}(\Bbb{R}^{n})}&:=&\sup\limits_{\substack{h_1,\cdots,h_n>0\\x_{1},\cdots x_{n}\in\Bbb{R}}}\bigg(\frac{1}{h_{1}\cdots h_{n}}\int\limits_{x_{1}}^{x_{1}+h_{1}}\cdots\int\limits_{x_{n}}^{x_{n}+h_{n}}w(t_{1},\cdots,t_{n})dt_{1}\cdots dt_{n}\bigg)\\
&&\times\bigg(\frac{1}{h_{1}\cdots h_{n}}\int\limits^{x_{1}}_{x_{1}-h_{1}}\cdots\int\limits^{x_{n}}_{x_{n}-h_{n}}w^{1-p'}(t_{1},\cdots,t_{n})dt_{1}\cdots dt_{n}\bigg)^{p-1}<\infty;
\end{eqnarray*}
further, $w\in A_{p}^{-(s)}(\Bbb{R}^{n})$ if
\begin{eqnarray*}
\|w\|_{A_{p}^{+(s)}(\Bbb{R}^{n})}&:=&\sup\limits_{\substack{h_1,\cdots,h_n>0\\x_{1},\cdots x_{n}\in\Bbb{R}}}\bigg(\frac{1}{h_{1}\cdots h_{n}}\int\limits^{x_{1}}_{x_{1}-h_{1}}\cdots\int\limits^{x_{n}}_{x_{n}-h_{n}}w(t_{1},\cdots,t_{n})dt_{1}\cdots dt_{n}\bigg)\\
&&\times\bigg(\frac{1}{h_{1}\cdots h_{n}}\int\limits_{x_{1}}^{x_{1}+h_{1}}\cdots\int\limits_{x_{n}}^{x_{n}+h_{n}}w^{1-p'}(t_{1},\cdots,t_{n})dt_{1}\cdots dt_{n}\bigg)^{p-1}<\infty.
\end{eqnarray*}

\begin{defin}
Let $1<p<\infty$. A weight function $w=w(x_1,\cdots,x_n)$ defined on ${\Bbb{R}}^n$ is said to satisfy $A_p^{-}$ condition in $x_{i}$ uniformly with respect to other variables $(w\in A_{p}^{-}(x_i))$ if
\begin{eqnarray*}
\|w\|_{A^{-}_{p}(x_i)}\equiv \esssup\limits_{\substack {(x_{1},\cdots x_{i-1},\\x_{i+1}\cdots,x_{n})\in{\Bbb{R}}^{n-1}}}\sup_{h_{i}>0} \bigg(\frac{1}{h_{i}}\int\limits_{x_{i}}^{x_{i}+h_i}w(x_1,\cdots,x_{i-1}, t, x_{i-1}, \cdots, x_n)dt\bigg)\\
\times\bigg(\frac{1}{h_{i}}\int\limits_{x_{i}-h_i}^{x_i}w(x_1,\cdots,x_{i-1}, t, x_{i-1}, \cdots, x_n)^{-1/(p-1)}dt \bigg)^{p-1}<\infty.
\end{eqnarray*}
further, $w\in A_{p}^{+}(x_i)$ if
\begin{eqnarray*}
\|w\|_{A^{+}_{p}(x_i)}\equiv \esssup\limits_{\substack {(x_{1},\cdots x_{i-1},\\x_{i+1}\cdots,x_{n})\in{\Bbb{R}^{n-1}}}}\sup_{h_{i}>0} \bigg(\frac{1}{h_{i}}\int\limits_{x_{i}-h_{i}}^{x_{i}}w(x_1,\cdots,x_{i-1}, t, x_{i-1}, \cdots, x_n)dt \bigg)\\
\times\bigg(\frac{1}{h_{i}}\int\limits_{x_{i}}^{x_i+h_i}w(x_1,\cdots,x_{i-1}, t, x_{i-1}, \cdots, x_n)^{-1/(p-1)}dt\bigg)^{p-1}<\infty.
\end{eqnarray*}
\end{defin}

{\em Remark.} It is known that (see \cite{KOK-MES-PER}, Ch. 5) that $w(x_1,\cdots,x_n)\in A_{p}^{\pm(s)}\Leftrightarrow w\in\bigcap\limits_{i=1}^{n} A_{p}^{\pm}(x_i)$.

\begin{thm}\label{m58}
Let $1<p<\infty$.

\rm{(i)} Suppose that a weight function $w$ on ${\Bbb{R}}^n$ belongs to the class $A^{+(s)}_p({\Bbb{R}}^n)$. Then there exists a constant $c$ depending only on $n$ and $p$ such that the following inequality
\begin{equation}\label{m59}
\|M^{+(s)}f\|_{L^{p}_{w}(\Bbb{R}^n)}\le c\Bigg(\prod\limits_{i=1}^{n}\|w\|_{A_p^{+}(x_{i})}\Bigg)^{1/(p-1)}\|f\|_{L^{p}_{w}(\Bbb{R}^n)}
\end{equation}
holds for all $f\in L^{p}_w(\Bbb{R}^n)$ . Further, the power $1/(p-1)$ in estimate \eqref{m59} is sharp.

\rm{(ii)} Let $w \in A^{-(s)}_p({\Bbb{R}}^n)$. Then there exists a constant $c$ depending only on $n$ and $p$ such that the following inequality
\begin{equation}\label{m59'}
\|M^{-(s)}f\|_{L^{p}_{w}(\Bbb{R}^n)}\le c\Bigg(\prod\limits_{i=1}^{n}\|w\|_{A_p^{-}(x_{i})}\Bigg)^{1/(p-1)}\|f\|_{L^{p}_{w}(\Bbb{R}^n)}
\end{equation}
holds for all $f\in L^{p}_w(\Bbb{R}^n)$ . Further, the power $1/(p-1)$ in estimate \eqref{m59'} is sharp.
\end{thm}
\begin{proof} We show (i). The proof of (ii) is similar.  Since the proof of inequality \eqref{m59} follows in the same way as in the case of $M^{(s)}$ (see Theorem \ref{m14}), we show only sharpness. Let $n=2$. We take $0<\epsilon<1$. Let $w(x_{1},x_{2})=|1-x_{1}|^{(1-\epsilon)(p-1)}|1-x_{2}|^{(1-\epsilon)(p-1)}$. Then it is easy to check that $$(\|w\|_{A_{p}^{+}(x_1)}\|w\|_{A_{p}^{+}(x_2)})^{1/(p-1)}\approx\frac{1}{\epsilon^2}.$$
Observe also that for
$$f(x_{1},x_{2})={(1-x_{1})}^{\epsilon(p-1)-1}\chi_{(0,1)}(x_{1}){(1-x_{2})}^{\epsilon(p-1)-1}\chi_{(0,1)}(x_{2}),$$
we have $\|f\|_{L^p_w}\approx\frac{1}{\epsilon^2}$.
Now let $0<x_{1},x_{2}<1$. Then
$$M^{+(s)}f(x_{1},x_{2})\ge\frac{1}{(1-x_{1})(1-x_{2})}\int\limits^{1}_{x_{1}}\int\limits^{1}_{x_{2}}f(t,\tau)dtd\tau=c\frac{1}{\epsilon^{2}}f(x_{1},x_{2}).$$
Finally
$$\|M^{+(s)}f\|_{L^{p}_w(\Bbb{R}^2)}\ge c\frac{1}{\epsilon^2}\|f\|_{L^{p}_{w}}.$$
Thus we have the sharpness in  \eqref{m59}.
\end{proof}

\subsection{One--sided  Multiple Fractional Integrals}

Now we discuss sharp bounds for one--sided strong maximal potential operators with product kernels.

Let $f$ be a locally integrable function on $\Bbb{R}^{n}$ and let $0<\alpha<1$. We define one--sided strong fractional maximal operators as
\begin{equation}\label{m56}
M^{+(s)}_{\alpha}f(x_{1},\cdots,x_{n})=\sup_{h_1,\cdots,h_n>0} \frac{1}{\prod\limits_{i=1}^{n}{h_{i}^{1-\alpha}}}\int\limits_{x_{1}}^{x_{1}+h_{1}}\cdots\int\limits_{x_{n}}^{x_{n}+h_{n}}|f(y_1,\cdots y_n)|dy_1\cdots dy_{n},
\end{equation}
\begin{equation}\label{m57}
M^{-(s)}_{\alpha}f(x_{1},\cdots,x_{n})=\sup_{h_1,\cdots,h_n>0} \frac{1}{\prod\limits_{i=1}^{n}{h_{i}^{1-\alpha}}}\int\limits_{x_{1}-h_{1}}^{x_{1}}\cdots\int\limits_{x_{n}-h_n}^{x_{n}}|f(y_1,\cdots y_n)|dy_1\cdots dy_{n}.
\end{equation}

Let $1<p\leq q<\infty$. We say that a weight function $w$ belongs to the class $A_{p,q}^{-(s)}(\Bbb{R}^{n})$ if
\begin{eqnarray*}
\|w\|_{A_{p,q}^{-(s)}(\Bbb{R}^{n})}&:=&\sup\limits_{\substack{h_1,\cdots,h_n>0\\x_{1},\cdots x_{n}\in\Bbb{R}}}\bigg(\frac{1}{h_{1}\cdots h_{n}}\int\limits_{x_{1}}^{x_{1}+h_{1}}\cdots\int\limits_{x_{n}}^{x_{n}+h_{n}}w^{q}(t_{1},\cdots,t_{n})dt_{1}\cdots dt_{n}\bigg)^{1/q}\\
&&\times\bigg(\frac{1}{h_{1}\cdots h_{n}}\int\limits^{x_{1}}_{x_{1}-h_{1}}\!\!\!\cdots\!\!\!\int\limits^{x_{n}}_{x_{n}-h_{n}}w^{-p'}(t_{1},\cdots,t_{n})dt_{1}\cdots dt_{n}\bigg)^{1/p'}<\infty;
\end{eqnarray*}
further, $w\in A_{p,q}^{+(s)}(\Bbb{R}^{n})$ if
\begin{eqnarray*}
\|w\|_{A_{p,q}^{+(s)}(\Bbb{R}^{n})}&:=&\sup\limits_{\substack{h_1,\cdots,h_n>0\\x_{1},\cdots x_{n}\in\Bbb{R}}}\bigg(\frac{1}{h_{1}\cdots h_{n}}\int\limits^{x_{1}}_{x_{1}-h_{1}}\cdots\int\limits^{x_{n}}_{x_{n}-h_{n}}w^q(t_{1},\cdots,t_{n})dt_{1}\cdots dt_{n}\bigg)^{1/q}\\
&&\times\bigg(\frac{1}{h_{1}\cdots h_{n}}\int\limits_{x_{1}}^{x_{1}+h_{1}}\!\!\!\cdots\!\!\!\int\limits_{x_{n}}^{x_{n}+h_{n}}w^{-p'}(t_{1},\cdots,t_{n})dt_{1}\cdots dt_{n}\bigg)^{1/p'}<\infty.
\end{eqnarray*}

\begin{defin}
Let $1<p\leq q<\infty$. A weight function $w=w(x_1,\cdots,x_n)$ defined on ${\Bbb{R}}^n$ is said to satisfy $A_{p,q}^{-}$ condition in $x_{i}$ uniformly with respect to other variables $(w\in A_{p,q}^{+}(x_i))$ if
\begin{eqnarray*}
\|w\|_{A^{+}_{p,q}(x_i)}:=  \esssup\limits_{\substack {(x_{1},\cdots x_{i-1},\\x_{i+1}\cdots,x_{n})\in{\Bbb{R}}^{n-1}}}\sup_{h_{i}>0} \bigg(\frac{1}{h_{i}}\int\limits_{x_{i}}^{x_{i}+h_i}w^{q}(x_1,\cdots, x_{i-1}, t, x_{i+1}  \cdots, x_n)dt\bigg)^{1/q}\times
\end{eqnarray*}

$$\indent\indent\times\bigg(\frac{1}{h_{i}}\int\limits_{x_{i}-h_i}^{x_i}w^{-p'}(x_1,\cdots, x_{i-1}, t, x_{i+1} \cdots, x_n) dt \bigg)^{1/p'}<\infty,$$
further, $w\in A_{p,q}^{-}(x_i)$ if
\begin{eqnarray*}
\|w\|_{A^{-}_{p,q}(x_i)}\equiv \sup\limits_{\substack {(x_{1},\cdots x_{i-1},\\x_{i+1}\cdots,x_{n})\in{\Bbb{R}}^{n-1}}}\sup_{h_{i}>0} \bigg(\frac{1}{h_{i}}\int\limits_{x_{i}-h_{i}}^{x_{i}}w^{q}(x_1,\cdots, x_{i-1}, t, x_{i+1} \cdots, x_n)dt\bigg)^{1/q}\times
\end{eqnarray*}

$$\indent\indent\times\bigg(\frac{1}{h_{i}}\int\limits_{x_{i}}^{x_i+h_i}w^{-p'} (x_1,\cdots, x_{i-1}, t, x_{i+1} \cdots, x_n)dt\bigg)^{1/p'}<\infty.$$
\end{defin}

\begin{rem}  It is easy to check that $w(x_1,\cdots,x_n)\in A_{p,q}^{\pm(s)}\Leftrightarrow w\in \bigcap\limits_{i=1}^{n}A_{p,q}^{\pm}(x_i)$.
\end{rem}

\begin{thm}\label{m62}
Let $0<\alpha<1$, $1<p<1/\alpha$. We put $q=\frac{p}{1-\alpha p}$. Suppose that $w$ is a weight function defined on $\Bbb{R}^{n}$ such that $w\in A^{+(s)}_{p,q}({\Bbb{R}}^n)$. Then there exists a constant $c$ depending only on $n$, $p$ and $\alpha$ such that the following inequality
\begin{equation}\label{m63a}
\|wM_{\alpha}^{+(s)}f\|_{L^{q}({\Bbb{R}}^n)}\le c\big (\|w\|_{A_{p,q}^{+}(x_{1})}\cdots\|w\|_{A_{p,q}^{+}(x_{n})}\big)^{\frac{p'}{q}(1-\alpha)}\|wf\|_{L^{p}({\Bbb{R}}^n)}
\end{equation}
holds for all $f\in L^{p}_{w^p}({\Bbb{R}}^n)$ . Further, the power ${\frac{p'}{q}(1-\alpha)}$ in estimate \eqref{m63a} is sharp.
\end{thm}
\begin{proof}
Estimate \eqref{m63a} follows in the same way as in the previous cases.
For sharpness we take  $n=2$ and $w(x_{1},x_{2})=w_1(x_1) w_2(x_2)$, where $w_1(x_1)= |1-x_{1}|^{(1-\epsilon)p'}$; $w_2(x_2)= |1-x_{2}|^{(1-\epsilon)p'}$,   $0<\epsilon<1$. Then
$$ \|w\|_{A^+_{p,q}(x_1)} \|w\|_{A^+_{p,q}(x_2)} = \prod\limits_{i=1}^2\|w_i\|_{A^+_{p,q}({\Bbb{R}})} =
\prod\limits_{i=1}^{2}\|w_i^q\|_{A^+_{1+q/p'}({\Bbb{R}})} \approx \varepsilon^{2q/p'}.$$

If $$f(t_{1},t_{2})=(1-t_{1})^{\epsilon-1}\chi_{(0,1)}(t_{1})(1-t_{2})^{\epsilon-1}\chi_{(0,1)}(t_{2}),$$
then $\|wf\|_{L^p({\Bbb{R}})^2}\approx \frac{1}{\epsilon^{2/p}}$.
Now let $0<x<1$. Then we find that the following estimate
$$M^{+(s)}_{\alpha}f(x_{1},x_{2})\ge\frac{1}{\prod\limits_{i=1}^{2}|1-x_{i}|^{1-\alpha}}\int\limits^{1}_{x_{1}}\int\limits^{1}_{x_{2}}f(t_{1},t_{2})dt_{1}dt_{2}
\approx \frac{\prod\limits_{i=1}^{2}|1-x_{i}|^{\epsilon-1+\alpha}}{\epsilon^{2}}$$
holds. Finally
\begin{equation}\label{m63}
\|wM^{+(s)}_{\alpha}f\|_{L^{q}({\Bbb{R}}^2)}\ge \epsilon^{-2-2/q}.
\end{equation}
Thus, letting $\epsilon\rightarrow0$ we are done.
\end{proof}

The next statement can be proved analogously. Details are omitted.
\begin{thm}\label{m63b}
Let  $\alpha$, $p$ and $q$satisfy the condition of Theorem \ref{m62}. Let $w$ be a weight function on $\Bbb{R}^{n}$ such that $w\in{A^{-(s)}_{p,q}(\Bbb{R}^{n})}$. Then there exist a constant $c$ depending only on $n$, $p$ and $\alpha$ such that the following inequality
\begin{equation}\label{m64}
\|wM_{\alpha}^{-(s)}f\|_{L^{q}(\Bbb{R}^{n})}\le c(\|w\|_{A_{p,q}^{-}(x_{1})}\cdots\|w\|_{A_{p,q}^{+}(x_{n})})^{\frac{p'}{q}(1-\alpha)}\|wf\|_{L^{p}(\Bbb{R}^{n})}
\end{equation}
holds for all $f\in L^{p}_{w^p}(\Bbb{R}^{n})$ . Further, the power ${\frac{p'}{q}(1-\alpha)}$ in estimate \eqref{m64} is sharp.
\end{thm}

Let $f$ be a measurable function on $\Bbb{R}^{n}$ and let $0<\alpha<1$. We define one--sided potentials $\mathcal{R}_{\alpha}^{(n)}$ and $\mathcal{W}_{\alpha}^{(n)}$ with product kernels
$$\mathcal{R}_{\alpha}^{(n)}f(x_{1},\cdots,x_n)=\int\limits_{-\infty}^{x_1}\cdots\int\limits_{-\infty}^{x_n}\frac{f(t_{1},\cdots,t_n)}{(x_{1}-t_{1})
^{1-\alpha}\cdots(x_{n}-t_{n})^{1-\alpha}}dt_1\cdots dt_n,$$
$$\mathcal{W}_{\alpha}^{(n)}f(x_{1},\cdots,x_n)=\int\limits^{\infty}_{x_1}\cdots\int\limits^{\infty}_{x_n}\frac{f(t_{1},\cdots,t_n)}{(t_{1}-x_{1})
^{1-\alpha}\cdots(t_{n}-x_{n})^{1-\alpha}}dt_1\cdots dt_n,$$
where $x_i\in \Bbb{R}$, $i=1,\cdots, n$.
\vskip+0.2cm

Finally we formulate the ''sharp result'' for one--sided potentials with product kernels. We do not repeat the arguments using above, and therefore omit the proof of the next statement.

\begin{thm}\label{m74}
Let $\alpha$, $p$ and $q$ satisfy the conditions of Theorem \ref{m62}. Suppose that $w$ be a weight function on $\Bbb{R}^{n}$ such that $w\in{A^{-(s)}_{p,q}({\Bbb{R}}^n)}$. Then

\rm{(i)} there exists a constant $c$ depending only on $n$, $p$ and $\alpha$ such that the following inequality
\begin{equation}\label{m75}
\|w\mathcal{R}^{(n)}_{\alpha}f\|_{L^{q}({\Bbb{R}}^n)}\le c\Bigg(\prod\limits_{i=1}^{n}\|w\|_{A_{p,q}^{-}(x_{i})}\Bigg)^{\max\{1,\frac{p'}{q}\}(1-\alpha)}\|wf\|_{L^{p}({\Bbb{R}}^n)}
\end{equation}
holds for all $f\in L^{p}_{w^p}({\Bbb{R}}^n)$ . Further, the power ${\max\{1,\frac{p'}{q}\}(1-\alpha)}$ in estimate \eqref{m75} is sharp.

\rm{(ii)} There is a constant $c$ depending only on $n$, $p$ and $\alpha$ such that
\begin{equation}\label{m76}
\|w\mathcal{W}^{(n)}_{\alpha}f\|_{L^{q}({\Bbb{R}}^n)}\le c\Bigg(\prod\limits_{i=1}^{n}\|w\|_{A_{p,q}^{+}(x_{i})}\Bigg)^{\max\{1,\frac{p'}{q}\}(1-\alpha)}\|wf\|_{L^{p}({\Bbb{R}}^n)}
\end{equation}
for all $f\in L^{p}_{w^p}({\Bbb{R}}^n)$. Further, the power ${\max\{1,\frac{p'}{q}\}(1-\alpha)}$ in estimate \eqref{m76} is sharp.
\rm{(ii)} There is a constant $c$ depending only on $n$, $p$ and $\alpha$ such that
\begin{equation}\label{m76}
\|w\mathcal{W}^{(n)}_{\alpha}f\|_{L^{q}({\Bbb{R}}^n)}\le c\Bigg(\prod\limits_{i=1}^{n}\|w\|_{A_{p,q}^{+}(x_{i})}\Bigg)^{\max\{1,\frac{p'}{q}\}(1-\alpha)}\|wf\|_{L^{p}({\Bbb{R}}^n)}
\end{equation}
for all $f\in L^{p}_{w^p}({\Bbb{R}}^n)$. Further, the power ${\max\{1,\frac{p'}{q}\}(1-\alpha)}$ in estimate \eqref{m76} is sharp.
\end{thm}
\subsection*{Acknowledgements}
The first and second authors were partially supported by  the Shota Rustaveli National
Science Foundation Grant (Contract Numbers:  D/13-23 and 31/47). The  third author is  thankful to the Higher Education Commission, Pakistan for the financial support.

\end{document}